\documentclass[aap]{imsart}
\usepackage{amsmath,amssymb,amsthm,enumerate,natbib,color,bbm,ifthen,graphicx,overpic}
\usepackage[latin1]{inputenc}
\usepackage[draft]{fixme}

\def\Nset{{\mathbb{N}}}
\def\Rset{\mathbb R}

\def\C0{\mathsf{C}_0}

\newcommand{\eg}{{\em e.g.} }
\newcommand{\as}{\text{a.s.} }

\newcommand{\eqdef}{\ensuremath{\stackrel{\mathrm{def}}{=}}}
\def\eqsp{\;}

\newcounter{rmnum}

\newcommand{\un}{\ensuremath{\mathbbm{1}}}
\def\Id{\mathrm{I}}

\newcommand{\gausspdf}[2]{\mathcal{N}(#1,#2)}
\newcommand{\ratio}{r}
\newcommand{\rate}[2][]
{\ifthenelse{\equal{#1}{}}{\ratesymbol(#2)}{\ratesymbol^{(#1)}(#2)}}

\def\Xset{\mathsf{X}} 
\def\Xsigma{\mathcal{X}} 

\def\L{\mathcal{L}}
\def\Lsub{\mathcal{M}}

\def\t{\theta}

\def\PE{\ensuremath{\mathbb E}}
\def\PP{\ensuremath{\mathbb P}}
\def\F{\mathcal{F}}

\newcommand{\tvnorm}[1]{\ensuremath{\left\|#1\right\|_{\mathrm{TV}}}}
\newcommand{\fnorm}[2]{\ensuremath{\left|#1\right|_{#2}}}

\newcommand{\supnorm}[1]{| #1 |_{\infty}}

\newcommand{\Vnorm}[2]{\left\| #1 \right\|_{#2}}

\newcommand{\dlim}{\ensuremath{\stackrel{\mathcal{D}}{\longrightarrow}}}
\newcommand{\plim}{\ensuremath{\stackrel{\PP}{\longrightarrow}}}
\newcommand{\aslim}{\ensuremath{\stackrel{\text{a.s.}}{\longrightarrow}}}

\newtheorem{theo}{Theorem}[section]
\newtheorem{lemma}[theo]{Lemma}

\newtheorem{prop}[theo]{Proposition}

\theoremstyle{remark}

\newcounter{hypA}
\newenvironment{hypA}{\refstepcounter{hypA}\begin{itemize}
  \item[{\bf A\arabic{hypA}}]}{\end{itemize}}

\newcounter{hypEE}
\newenvironment{hypEE}{\refstepcounter{hypEE}\begin{itemize}
  \item[{\bf I\arabic{hypEE}}]}{\end{itemize}}

\newcounter{hypAM}

\newcounter{hypU}

\def\Yproc{\ensuremath{\{Y_n \eqsp, n\geq 0\}}}

\def\rmd{\mathrm{d}}
\def\rme{\mathrm{e}}
\def\rmi{\mathrm{i}}
\def\1{\mathbbm{1}}
\def\Tset{\Theta}
\def\Tsigma{\mathcal{T}}

\newcommand{\CPE}[3][]
{\ifthenelse{\equal{#1}{}}{\mathbb{E}\left[\left. #2 \, \right| #3 \right]}{\mathbb{E}_{#1}\left[\left. #2 \, \right | #3 \right]}}
\newcommand{\CPP}[3][]
{\ifthenelse{\equal{#1}{}}{\mathbb{P}\left[\left. #2 \, \right| #3 \right]}{\mathbb{P}_{#1}\left(\left. #2 \, \right | #3 \right)}}

\newcommand{\coint}[1]{\left[#1\right)}
\newcommand{\ocint}[1]{\left(#1\right]}

\newcommand{\ccint}[1]{\left[#1\right]}

\def\operpoisson{\Lambda}

\newcounter{hypoconbis}
\newcounter{saveconbis}
\newcommand\debutA{\begin{list} {\textbf{A\arabic{hypoconbis}}}{\usecounter{hypo
conbis}}\setcounter{hypoconbis}{\value{saveconbis}}}
\newcommand\finA{\end{list}\setcounter{saveconbis}{\value{hypoconbis}}}
\newcommand{\chunk}[4][]%
{\ifthenelse{\equal{#1}{}}{\ensuremath{{#2}_{#3:#4}}}{\ensuremath{#2}_{#1,#3:#4}}
}
\newcommand{\Var}[2][]{\ifthenelse{\equal{#1}{}}{\mathrm{Var} \left[ #2 \right]}{\mathrm{Var}_{#1} \left[ #2 \right]}}

\newcommand{\sequence}[2]{ \left( #1_#2 \right)_{#2 \in \Nset}}
\newcommand{\sequencetwo}[3]{ \left\{ \left( #1_#3, #2_#3 \right)\right\}_{#3 \in \Nset} }

\begin{document}

\begin{frontmatter}

\title{A Central Limit Theorem for Adaptive and Interacting Markov Chains}
\runtitle{A Central Limit Theorem for iMCMC}
           
\begin{aug}
  \author{\fnms{G.}
    \snm{Fort}\thanksref{t1,T1,m1}\ead[label=e1]{gersende.fort@telecom-paristech.fr}},
  \author{\fnms{E.}
    \snm{Moulines}\thanksref{T1,m2}\ead[label=e2]{eric.moulines@telecom-paristech.fr}},
   \author{\fnms{P.}  \snm{Priouret}\thanksref{m3}
    \ead[label=e3]{priouret@ccr.jussieu.fr}} \and
    \author{\fnms{P.} \snm{Vandekerkhove} \ead[label=e4]{pierre.vandek@univ-mlv.fr} \thanksref{m4} }

  \thankstext{t1}{Corresponding author}
  \thankstext{T1}{This
    work is partially supported by the French National Research Agency, under
    the program ANR-08-BLAN-0218 BigMC}
  \runauthor{G. Fort et al.}

  \affiliation{CNRS \& TELECOM ParisTech \thanksmark{m1}\thanksmark{m2}, Univ.
    Pierre et Marie Curie\thanksmark{m3}, Univ. Paris Est\thanksmark{m4}}

  \address{LTCI, TELECOM ParisTech \& CNRS, \\
    46 rue Barrault, \\
    75634 Paris Cedex 13, \\
    France \\ \printead{e1} \\ \printead{e2}}
  \address{LPMA, Univ. Pierre et Marie Curie, \\  Boîte courrier 188 \\ 75252 PARIS Cedex 05, France,\\
    \printead{e3}} \address{LAMA, Univ. Paris Est \\ 5 Bd Descartes
    Champs-sur Marne \\
    77454 Marne-la-Vallée Cedex 2, France.}

\end{aug}

\begin{abstract}
Adaptive and interacting Markov Chains Monte Carlo (MCMC) algorithms are a novel class of non-Markovian algorithms aimed at improving
the simulation efficiency for complicated target distributions. In this paper, we study a general (non-Markovian) simulation framework 
covering both the adaptive and interacting MCMC algorithms. We establish a Central Limit Theorem for additive functionals of unbounded 
functions under a set of verifiable conditions, and identify the asymptotic variance.
Our result extends all the results reported so far.
An application to the interacting tempering algorithm (a simplified version of the equi-energy sampler) is presented to
support our claims.
\end{abstract}



\end{frontmatter}

\section{Introduction}
\label{sec:Introduction}
Markov chain Monte Carlo (MCMC) methods generate samples from
distributions  known up to a scaling factor.

In the last decade, several non-Markovian simulation algorithms have been
proposed.  In the so-called adaptive MCMC algorithm, the transition kernel of
the MCMC algorithm depends on a finite dimensional \emph{parameter} which is
updated at each iteration from the past values of the chain and the parameters.
The prototypical example is the adaptive Metropolis algorithm, introduced in
\cite{haario:saksman:tamminen:1999} (see \cite{saksman:vihola:2010} and the
references therein for recent references).  Many other examples of adaptive
MCMC algorithms are presented in the survey papers by \cite{andrieu:thoms:2008,rosenthal:2009,atchade:fort:moulines:priouret:2011}.

In the co-called \emph{Interacting MCMC}, several processes are simulated in
parallel, each targeting different distribution.  Each process might interact
with the whole past of its neighboring processes.  A prototypical example is the
equi-energy sampler introduced in \cite{kou:zhou:wong:2006}, where the
different processes target a tempered version of the target distribution. The
convergence of this algorithm has been considered in a series of papers by
\cite{andrieu:jasra:doucet:delmoral:2007},
\cite{andrieu:jasra:doucet:delmoral:2007b},
\cite{andrieu:jasra:doucet:delmoral:2008} and in
\cite{fort:moulines:priouret:2010}.  Different variants of the interacting MCMC
algorithm have been later introduced and studied in
\cite{bercu:delmoral:doucet:2009}, \cite{delmoral:doucet:2010} and
\cite{brockwell:delmoral:doucet:2010}. These algorithms are so far limited to
specific scenarios, and the assumptions used in these papers preclude the
applications of their results in the applications considered in this paper.

The analysis of the convergence of these algorithms is involved. Whereas the basic building blocks 
of these simulation algorithms are Markov kernels, the processes generated by these techniques 
are no longer Markovian. Indeed, each individual process either interacts with its distant past, or the distant 
past of some auxiliary processes.  

The ergodicity and the consistency of additive functionals for adaptive and interacting Markov Chains have been
considered in several recent papers: see \cite{fort:moulines:priouret:2010} and
the references therein. Up to now, there are much fewer works addressing
Central Limit Theorems (CLT).  In \cite{andrieu:moulines:2006} the authors
establish the asymptotic normality of additive functionals for a special class
of adaptive MCMC algorithms in which a finite dimensional parameter is adapted using a
stochastic approximation procedure. Some of the theoretical limitations of \cite{andrieu:moulines:2006} have
been alleviated by \cite{saksman:vihola:2010} for the so-called adaptive Metropolis algorithm, which established a CLT for
additive functionals for the Adaptive Metropolis algorithm (with a proof
specially tailored for this algorithm). The results presented in this
contribution contain as special cases these two earlier results.

The theory for interacting MCMC algorithms is up to now quite limited, despite the clear
potential of this class of methods to sample complicated multimodal target
distributions. The law of large numbers for additive functionals have been
established in \cite{andrieu:jasra:doucet:delmoral:2008b} for some specific
interacting algorithm.  A wider class of interacting Markov chains has been considered in
\cite{delmoral:doucet:2010}. This paper establishes the consistency of a form of interacting 
tempering algorithm and provides non-asymptotic $L^p$-inequalities.
 The assumptions under which the results are derived are restrictive and the results do not cover the
interacting MCMC algorithms considered in this paper. More recently, \cite{fort:moulines:priouret:2010} 
have established the ergodicity and law of large numbers for a wide class of interacting MCMC, under 
the weakest conditions known so far.  

A functional CLT was derived in \cite{bercu:delmoral:doucet:2009} for a
specific class of interacting Markov Chains but their assumptions do not cover
the interactive MCMC considered in this paper (and in particular, the
interacting MCMC algorithm).  A CLT for additive functionals is established by
\cite{atchade:2010} for the interacting tempering algorithm; the proof of the
main result in this paper, Theorem~3.3, contains a serious gap (p.865) which
seems difficult to correct.

This paper aims at providing a theory removing the limitations mentioned above
and covering both adaptive and interacting MCMC in a common unifying framework.
The paper is organized as follows.  In Section~\ref{sec:mainresult} we
establish CLTs for adaptive and interacting MCMC algorithms.  These results are
applied in section~\ref{sec:applicationIT} to the interacting tempering
algorithm which is a simplified version of the Equi-Energy sampler.  All the
proofs are postponed in Section~\ref{sec:proofs}.

\subsection*{Notations}
Let $(\Xset, \Xsigma)$ be a general state space and $P$ be a Markov transition
kernel (see e.g. \cite[Chapter 3]{meyn:tweedie:2009}). $P$ acts on bounded
functions $f$ on $\Xset$ and on $\sigma$-finite positive measures $\mu$ on
$\Xsigma$ via
\[
P f(x) \eqdef \int P(x,\rmd y) f(y) \eqsp, \qquad \mu P (A) \eqdef \int
\mu(\rmd x) P(x,A) \eqsp.
\]
We denote by $P^n$ the $n$-iterated transition kernel defined inductively 
\[
P^n(x,A) \eqdef \int P^{n-1}(x,\rmd y) P(y,A) = \int P(x,\rmd y) P^{n-1}(y,A)
\eqsp;
\]
where $P^0$ is the identity kernel.  For a function $V : \Xset \to \coint{1,+\infty}$, define the $V$-norm of a function $f: \Xset \to \Rset$ by
\[
\fnorm{f}{V} \eqdef \sup_{x \in \Xset} \frac{|f|(x)}{V(x)} \eqsp.
\]
When $V=1$, the $V$-norm is the supremum norm  denoted by
$\supnorm{f}$.  Let $\L_V$ be the set of measurable functions such that $\fnorm{f}{V} <
+\infty $. For $\mu$ a signed  measure on $(\Xset,\Xsigma)$, we defined $\Vnorm{\mu}{V}$ the $V$-norm of $\mu$ 
as 
$$
\Vnorm{\mu}{V}= \sup_{f \in \L_V, \fnorm{f}{V} \leq 1} |\mu(f)| \eqsp.
$$ 
When $V\equiv 1$, the $V$-norm corresponds to the
total variation norm.

For two transition kernels $P_1, P_2$, define the $V$-distance as 
\[
\Vnorm{P_1 -P_2}{V} \eqdef \sup_{x \in \Xset} V^{-1}(x) \Vnorm{P_1(x,\cdot)-P_2(x,\cdot)}{V} \eqsp.
\]
Let $\sequence{x}{n}$ a sequence. For $p \leq q \in \Nset^2$, $\chunk{x}{p}{q}$ denotes the vector $(x_p,\dots,x_q)$.

\section{Main results}
\label{sec:mainresult}
Let $(\Tset, \Tsigma)$ be a measurable space.  Let $\{P_\t, \t \in \Theta\}$ be
a collection of Markov transition kernels on $(\Xset,\Xsigma)$ indexed by a
parameter $\t \in \Theta$.  In the sequel, it is assumed that for any $A \in
\Xsigma$, $(x,\theta) \mapsto P_\theta(x,A)$ is $\Xsigma \otimes \Tsigma/
\mathcal{B}([0,1])$ measurable, where $\mathcal{B}([0,1])$ denotes the Borel
$\sigma$-field.  In the sequel $\Theta$ is not necessarily a finite-dimensional
vector space.  It might be a function space or a space of measures.  We
consider a $\Xset \times \Tset$-valued process $\sequencetwo{X}{\t}{n}$ on a
filtered probability space $(\Omega, \mathcal{A}, \{\F_n, n\geq 0\}, \PP)$.  It
is assumed that
\begin{hypA}
 \label{hyp:DefinitionProc:X:Theta} The process $\sequencetwo{X}{\t}{n}$ is
 $\sequence{\F}{n}$-adapted and for any bounded measurable function $h$,
\begin{equation*}
  \label{eq:DefinitionProc:X:Theta}
\CPE{h(X_{n+1})}{\F_n} = P_{\t_n} h(X_n) \eqsp.
\end{equation*}
\end{hypA}
Assumption \textbf{A\ref{hyp:DefinitionProc:X:Theta}} implies that conditional to the past
(subsumed in the $\sigma$-algebra $\F_n$), the distribution of the next sample
$X_{n+1}$ is governed by the current value $X_n$ and the current parameter
$\theta_n$. This assumption covers any adaptive and interacting MCMC
algorithms; see \cite{andrieu:thoms:2008},
\cite{atchade:fort:moulines:priouret:2011}, \cite{fort:moulines:priouret:2010}
for examples. This assumption on the adaptation of the parameter
$\sequence{\t}{n}$ is quite weak since it only requires the parameter to be
adapted to the filtration. In practice, it frequently occurs that the joint
process $\sequencetwo{X}{\t}{n}$ is Markovian but assumption
\textbf{A\ref{hyp:DefinitionProc:X:Theta}} covers more general adaptation rules.

We assume that the transition kernels $\{P_\t, \t \in \Theta\}$ satisfy a
Lyapunov drift inequality and smallness conditions:
\begin{hypA}
\label{hyp:geometric-ergodicity}
For all $\t \in \Theta$, $P_\t$ is phi-irreducible, aperiodic and there exists a
function $V: \Xset \to \coint{1, +\infty}$, and for any $\t \in
\Theta$ there exist some constants $b_\t \in (1, +\infty), \lambda_\t \in
(0,1)$ such that
\begin{equation*}
P_\t V \leq  \lambda_\t V + b_\t \eqsp.
\end{equation*}
In addition, for any $d \geq 1$ and any $\t \in \Tset$, the level sets $\{ V
\leq d \}$ are $1$-small for $P_\theta$.
\end{hypA}
In many examples considered so far (see \cite{andrieu:moulines:2006},
\cite{saksman:vihola:2010}, \cite{fort:moulines:priouret:2010},
\cite{andrieu:jasra:doucet:delmoral:2008}) this condition is satisfied.  All
the results below can be established under assumptions insuring that the drift
inequality and/or the smallness condition are satisfied for some $m$-iterated
$P^m_\theta$. Note that checking assumption on the iterated kernel $P_\theta^m$
is prone to be difficult because the expression of the $m$-iterated kernel is
most often rather involved.

\textbf{A\ref{hyp:geometric-ergodicity}} implies that, for any $\t \in \Theta$, $P_\t$
possesses an invariant probability distribution $\pi_\t$ and the kernel $P_\t$
is geometrically ergodic~\cite[Chapter 15]{meyn:tweedie:2009}. The following
lemma summarizes the properties of the family $\{P_{\t}, \t \in \Tset\}$
used in  the sequel (see e.g.~\cite{douc:moulines:rosenthal:2004} and references therein).
\begin{lemma}
\label{lem:BoundCandRho}
Assume \textbf{A\ref{hyp:geometric-ergodicity}}. Then for any $\t \in \Theta$, there
exists a probability distribution $\pi_\t$ such that $\pi_\t P_\t = \pi_\t$ and
$\pi_\t(V) \leq b_\t (1-\lambda_\t)^{-1}$. In addition, for any $\alpha \in
\ocint{0,1}$, the following property holds.
\begin{list}{}{}
\item {\bf P[$\alpha$]} For any $\theta \in \Theta$, there exist $C_\t < \infty$ and $\rho_\t \in (0,1 )$
  such that, for any $\gamma \in [\alpha,1]$,
$$\Vnorm{P_\t^n - \pi_\t}{V^\gamma} \leq C_\t \ \rho_\t^n\eqsp.
$$
\end{list}

\end{lemma}
Set \begin{equation}
\label{eq:DefinitionLtheta}
L_\theta  \eqdef C_\t \vee (1-\rho_\t)^{-1} \eqsp.
\end{equation}
It has been shown in \cite{fort:moulines:priouret:2010}, that under appropriate
assumptions, when the sequence $\sequence{\theta}{k}$ converges to
$\theta_\star \in \Tset$ in an appropriate sense, $n^{-1} \sum_{k=1}^n f(X_k)$
converges almost surely to $\pi_{\t_\star}(f)$, for any functions $f$ belonging
to a suitable class of functions $\Lsub$.

The objective of this paper is to derive a CLT for $
n^{-1/2} \sum_{k=1}^n \left\{ f(X_k) - \pi_{\t_\star}(f) \right\}$
for functions$f$ belonging to $\Lsub$. To that goal, consider the following decomposition
\[
n^{-1/2} \sum_{k=1}^n \left\{ f(X_k) - \pi_{\t_\star}(f) \right\}= S_n^{(1)}(f) + S_n^{(2)}(f) \eqsp,
\]
where $S_n^{(1)}(f)$ and $S_n^{(2)}(f)$ are given by
\begin{align}
\label{eq:definition-S_1}
& S_n^{(1)}(f) \eqdef n^{-1/2} \sum_{k=1}^n \left\{ f(X_k) - \pi_{\t_{k-1}}(f) \right\} \eqsp, \\
\label{eq:definition-S_2}
& S_n^{(2)}(f) \eqdef n^{-1/2} \sum_{k=0}^{n-1} \left\{ \pi_{\t_{k}}(f) -
  \pi_{\t_\star}(f) \right\} \eqsp.
\end{align}
We consider these two terms separately.  For the first term, we use a classical technique 
based on the Poisson decomposition; this amounts to write
$S_n^{(1)}(f)$ as the sum of a martingale difference and of a remainder term
converging to zero in probability; see
\cite{andrieu:moulines:2006,atchade:fort:2010,fort:moulines:priouret:2010,delmoral:doucet:2010,saksman:vihola:2010}
for law of large numbers for adaptive and interacting MCMC).  Then we apply a
classical CLT for martingale difference array; see for example \cite[Theorem
3.2]{hall:heyde:1980}.

The second term vanishes when $\pi_\t= \pi_{\theta_\star}$ for all $\t \in \Tset$ which is the case for example, for the adaptive Metropolis
  algorithm~\citep{haario:saksman:tamminen:1999}. In scenarios where $\t \mapsto \pi_\t$ is a non trivial function of
  $\theta$, the weak convergence  $S_n^{(2)}(f)$ relies on conditions which are quite problems specific.

  The application detailed in Section~\ref{sec:applicationIT}, an elementary
  version of the interacting tempering algorithm, is a situation in which
  $\pi_{\t_\star}$ is known but the expression of $\pi_\t$, $\t \neq \t_\star$,
  is unknown, except in very simple examples.  The Wang-Landau
  algorithm~\citep{wang:landau:2001,liang:liu:carroll:2007} is an example of
  adaptive MCMC algorithm in which $\t \mapsto \pi_\t$ is explicit.

  The results in this paper cover the case when the expression of $\pi_\t$ is unknown: we rewrite $S_n^{(2)}(f)$ by using a
  linearization of the fluctuation $\pi_{\t_k}(f) - \pi_{\t_\star}(f)$ in terms
  of the difference $P_{\t_k} - P_{\t_\star}$
\[
\pi_{\t_k}(f) - \pi_{\t_\star}(f) = \pi_{\t_\star} \left(P_{\t_k} -
  P_{\t_\star} \right) \operpoisson_{\t_\star}(f) +\Xi(f,\t_k) \eqsp.
\]
Our approach covers much more general set-up than the one outlined
in~\cite{bercu:delmoral:doucet:2009}.

By \textbf{A\ref{hyp:geometric-ergodicity}}, for any $\alpha \in (0,1)$ and $f \in
\L_{V^\alpha}$, the function $\sum_{n \geq 0} P_{\t}^n \left(f - \pi_{\t}(f)
\right)$ exists and is in $\L_{V^\alpha}$. For $\t \in \Theta$, denote by
$\operpoisson_{\t}: \L_{V^\alpha} \mapsto \L_{V^\alpha}$ the transition kernel
which associates to any function $f \in \L_{V^\alpha}$ the function $
\operpoisson_\t f$ given by:
\begin{equation}
  \label{eq:operateur:poisson}
  \operpoisson_{\t} f \eqdef \sum_{n \geq 0} P_{\t}^n f - \pi_{\t}(f)  \eqsp.
\end{equation}
The function $\operpoisson_{\t} f$ is the solution of the Poisson equation
\begin{equation}
  \label{eq:PoissonEquation}
  \operpoisson_{\t} f - P_\t \operpoisson_{\t} f = f - \pi_\t (f) \eqsp.
\end{equation}
This solution is unique up to an additive constant (see e.g.~\cite[Proposition
17.4.1.]{meyn:tweedie:2009}).

The convergence of $S_n^{(1)}(f)$ is addressed under the following assumptions
which are related to the regularity in the parameter $\t \in \Tset$ of the ergodic behavior of the
kernels $\{P_\t, \t \in \Tset \}$.
\begin{hypA}\label{hyp:cvgProba:series}
  There exist $\alpha \in (0,1/2)$ and a subset of measurable functions
  $\Lsub_{V^\alpha} \subseteq \L_{V^\alpha}$ satisfying the two following conditions
  \begin{enumerate}[(a)]
  \item \label{hyp:cvgProba:series:1} for any  $f
    \in \Lsub_{V^\alpha}$,
\[
n^{-1/2} \sum_{k=1}^n \fnorm{P_{\t_k} \operpoisson_{\t_k}f - P_{\t_{k-1}}
  \operpoisson_{\t_{k-1}}f}{V^{\alpha}} V^{\alpha}(X_k) \plim 0 \eqsp.
\]
\item \label{hyp:cvgProba:series:2} $ n^{-1/2\alpha} \sum_{k=0}^{n-1}
  L_{\theta_{k}}^{2/\alpha} \, P_{\theta_{k}}V(X_{k}) \plim 0$ where $L_\t$ is
  defined by (\ref{eq:DefinitionLtheta}) for the constants $C_{\t}, \rho_{\t}$
  given by P[$\alpha$].
  \end{enumerate}
\end{hypA}
\textbf{A\ref{hyp:cvgProba:series}-\ref{hyp:cvgProba:series:1}} controls the regularity
in the parameter $\t$ of the Poisson solution $\operpoisson_\t f$.
Lemma~\ref{lem:regularity-in-theta:poisson} in Appendix~\ref{appendix} is
useful to check \textbf{A\ref{hyp:cvgProba:series}-\ref{hyp:cvgProba:series:1}}. It
relates the regularity in $\t$ of the function $\t \mapsto P_\t \operpoisson_\t
f$ to the ergodicity constants $C_\t$ and $\rho_\t$ introduced in
Lemma~\ref{lem:BoundCandRho} and to the regularity in $\t$ of the function $\t
\mapsto P_\t$ from the parameter space $\Tset$ to the space of Markov
transition kernels equipped with the $V$-operator norm.

\textbf{A\ref{hyp:cvgProba:series}-\ref{hyp:cvgProba:series:2}} is a kind of containment
condition (see \cite{roberts:rosenthal:2007}): when the ergodic behavior
\textbf{A\ref{hyp:geometric-ergodicity}} is uniform in $\t$ so that $\lambda_\t$, $b_\t$
and the minorization constant of the $P_\t$-smallness condition do not depend
on $\t$, then the constant $L_\t$ does not depend on $\t$ and by
\textbf{A\ref{hyp:DefinitionProc:X:Theta}} and the drift inequality
\textbf{A\ref{hyp:geometric-ergodicity}},
\[
n^{-1/2\alpha} \sum_{k=0}^{n-1} \PE\left[V(X_{k+1}) \right] \leq
n^{1-1/2\alpha} \ \left\{ \PE\left[V(X_0)\right] + (1-\lambda)^{-1} b
\right\} \to 0 \eqsp.
\]
Therefore, condition \textbf{A\ref{hyp:cvgProba:series}-\ref{hyp:cvgProba:series:2}}
holds provided the ergodic constant $L_{\t_k}$ is controlled by a
slowly-increasing function of $k$. Lemma~\ref{lem:explicit:control} in
Appendix~\ref{appendix} provides sufficient conditions to obtain upper bounds
of $\t \mapsto L_\t$ in terms of the constants appearing in the drift
inequality \textbf{A\ref{hyp:geometric-ergodicity}}.

We finally introduce a condition allowing to obtain a closed-form expression for the asymptotic
variance of $S_n^{(1)}(f)$. For $\t \in \Tset$ and $f \in \L_{V^\alpha}$ define
\begin{equation}
\label{eq:definition-gtheta}
F_{\t} \eqdef  P_{\t} (\operpoisson_\t f)^2 - \left[ P_{\t} \operpoisson_\t f\right]^2 \eqsp.
\end{equation}
\begin{hypA} \label{hyp:existence:variance}
  For any $f \in \Lsub_{V^\alpha}$, $ n^{-1} \sum_{k=0}^{n-1} F_{\t_k}(X_k) \plim
  \sigma^2(f)$, where $\sigma^2(f)$ is a deterministic constant.
\end{hypA}
Assumption \textbf{A\ref{hyp:existence:variance}} is typically established by using  the Law of Large
Numbers (LLN) for adaptive and interacting Markov Chain derived in \cite{fort:moulines:priouret:2010}; see also Theorem~\ref{theo:WeakLLN} in
Appendix~\ref{appendix}. Under appropriate regularity conditions on the Markov kernels
$\{P_\t, \t \in \Tset \}$, it is proved that $n^{-1} \sum_{k=0}^{n-1}
\{F_{\t_k}(X_k) - \int \pi_{\t_k}(\rmd x) \, F_{\t_k}(x) \}$ converges in
probability to zero.  The second step consists in showing  that $n^{-1} \sum_{k=0}^{n-1}
\int \pi_{\t_k}(\rmd x) \, F_{\t_k}(x)$ converges to a (deterministic) constant
$\sigma^2(f)$: when $\pi_{\t}$ is not explicitly known and the set $\Xset$ is Polish,
Lemma~\ref{lem:mun:fn:convergence} in Appendix~\ref{appendix} is useful to
check this convergence. In practice, this may introduce a restriction of the
set of functions $f \in \L_{V^\alpha}$ for which this limit holds (see e.g. the
example detailed in Section~\ref{sec:applicationIT} where $\Lsub_{V^\alpha}
\neq \L_{V^\alpha}$).

We can now state conditions upon which $S_n^{(1)}(f)$ is asymptotically normal.
\begin{theo}
\label{theo:CLT:martingale}
Assume \textbf{A\ref{hyp:DefinitionProc:X:Theta}} to \textbf{A\ref{hyp:existence:variance}}. For
any $f \in \Lsub_{V^\alpha}$,
\[
\frac{1}{\sqrt{n}} \sum_{k=1}^n \left\{ f(X_k) -\pi_{\t_{k-1}} (f) \right\}\dlim
\gausspdf{0}{\sigma^2(f)} \eqsp.
\]
\end{theo}
The proof is in section~\ref{sec:proof:theo:CLT:martingale}.  When $\pi_\t =
\pi$ for any $\t$, Theorem~\ref{theo:CLT:martingale} provides sufficient
conditions for a CLT for additive functionals to hold.

When $\pi_\t$ is a function of $\t \in \Tset$, we need now to obtain a joint
CLT for $(S_n^{(1)}(f),S_n^{(2)}(f))$ (see \eqref{eq:definition-S_1} and
\eqref{eq:definition-S_2}).  To that goal, we replace
\textbf{A\ref{hyp:DefinitionProc:X:Theta}} by the following assumption which implies
that, conditionally to the process $\sequence{\theta}{k}$, $\sequence{X}{k}$ is an inhomogeneous Markov chain with transition kernels $(P_{\theta_j},
j \geq 0)$:
\begin{hypA}
\label{hyp:X_conditionellement_markov}
There exists an initial distribution $\nu$ such that for any bounded measurable
function $f: \Xset^{n+1} \to \Rset$,
\[
\CPE{f(\chunk{X}{0}{n})}{\t_{0:n}}= \idotsint \nu(\rmd x_0) f\left(
  \chunk{x}{0}{n} \right) \prod_{j=1}^n P_{\t_{j-1}}(x_{j-1}, \rmd x_j) \eqsp.
\]
\end{hypA}
Assumption  \textbf{A\ref{hyp:X_conditionellement_markov}} is satisfied when $\sequencetwo{X}{\t}{n}$ is an interacting MCMC algorithm.
Note that \textbf{A\ref{hyp:X_conditionellement_markov}} implies \textbf{A\ref{hyp:DefinitionProc:X:Theta}}.

The first step in the proof of the joint CLT consists in linearizing the
difference $\pi_{\t_n} - \pi_{\t_\star}$.  Under
\textbf{A\ref{hyp:geometric-ergodicity}}, $\pi_\t(g)$ exists for any $g \in
\L_{V^\alpha}$ and $\t\in \Theta$ (see Lemma~\ref{lem:BoundCandRho}), and we
have
  \[
  \pi_\t(g) - \pi_{\t_\star}(g) = \pi_\t P_\t g - \pi_{\t_\star}
  P_{\t_\star} g  = \pi_\t \left(P_\t - P_{\t_\star} \right) g + \left( \pi_\t
    - \pi_{\t_\star} \right) P_{\t_\star} g \eqsp,
  \]
  which implies that $\left(\pi_\t- \pi_{\t_\star} \right)\left( \Id -
    P_{\t_\star} \right) g = \pi_\t \left(P_\t - P_{\t_\star} \right) g$.  Let
  $f \in \L_{V^\alpha}$. Then $\operpoisson_{\t_\star} f \in \L_{V^\alpha}$ and
  by applying the previous equality with $g = \operpoisson_{\t_\star} f$, we
  have by (\ref{eq:PoissonEquation})
\begin{equation}
  \label{eq:DL:pi:ordre1}
  \pi_\t(f) - \pi_{\t_\star}(f) = \pi_\t \left(P_\t - P_{\t_\star}
\right)\operpoisson_{\t_\star} f \eqsp.
\end{equation}
We can iterate this decomposition, writing 
\begin{multline*}
\pi_\t(f) - \pi_{\t_\star}(f)= \pi_{\t_\star}  \left(P_\t - P_{\t_\star} \right)\operpoisson_{\t_\star} f
+ 
\\\pi_{\t}\left( \left(P_\t - P_{\t_\star} \right)\operpoisson_{\t_\star} f\right) - \pi_{\t_\star}\left( \left(P_\t - P_{\t_\star} \right)\operpoisson_{\t_\star} f\right)
\end{multline*}
Applying again  (\ref{eq:DL:pi:ordre1}), we obtain
\begin{equation*}
\pi_\t(f) - \pi_{\t_\star}(f) =  \pi_{\t_\star} \left(P_\t - P_{\t_\star} \right)\operpoisson_{\t_\star} f + \pi_\t \left(P_\t - P_{\t_\star} \right)\operpoisson_{\t_\star}\left(P_\t - P_{\t_\star}
  \right)\operpoisson_{\t_\star} f \eqsp.
\end{equation*}
This decomposition can be iterated, which yields
The first term in the RHS of the previous equation is the leading term of the
error $\pi_{\t_k} - \pi_{\t_\star}$, whereas the second term is a remainder.
This decomposition naturally leads to the following assumption.
\begin{hypA}
\label{hyp:variance:TCLsurPi}
For any function $f \in \Lsub_{V^\alpha}$,
 \begin{enumerate}[(a)]
 \item \label{hyp:variance:TCLsurPi:lineaire} there exists a positive constant
   $\gamma^2(f)$ such that
\begin{equation}
\label{eq:definition-gamma}
n^{-1/2} \sum_{k=1}^n     \pi_{\t_\star} \left(P_{\t_k} - P_{\t_\star} \right) \operpoisson_{\t_\star} f
\dlim \gausspdf{0}{\gamma^2(f)} \eqsp.
\end{equation}
\item  \label{hyp:variance:TCLsurPi:reste}
$n^{-1/2} \sum_{k=1}^n \pi_{\t_k} \left(P_{\t_k} - P_{\t_\star} \right)\operpoisson_{\t_\star}
\left(P_{\t_k} - P_{\t_\star}  \right)\operpoisson_{\t_\star} f
\plim 0 $.
\end{enumerate}
\end{hypA}

\begin{theo}
\label{theo:TCL}
Assume \textbf{A\ref{hyp:geometric-ergodicity}} to A\ref{hyp:variance:TCLsurPi}.  For
any function $f \in \Lsub_{V^\alpha}$,
\[
\frac{1}{\sqrt{n}} \sum_{k=1}^n \left\{ f(X_k) -\pi_{\t_\star} (f) \right\}\dlim
\mathcal{N}\left(0, \sigma^2(f) + \gamma^2(f) \right)\eqsp.
\]
\end{theo}
The proof of Theorem~\ref{theo:TCL} is postponed to section~\ref{sec:proof:theo:TCL}.
It is worthwhile to note that, as a consequence of \textbf{A\ref{hyp:X_conditionellement_markov}}, the variance is additive. 
This result extends \cite{bercu:delmoral:doucet:2009} which addresses the case
when $P_{\t}(x,A) = P_\t(A)$ \textit{i.e.} the case when conditionally to the
adaptation process $\sequence{\t}{n}$, the random variables $\sequence{X}{n}$
are independent (see \cite[Eq. (1.4)]{bercu:delmoral:doucet:2009}).
Our result, applied in this simpler situation, yields to the same asymptotic variance.

\section{Application to Interacting Tempering algorithm}
\label{sec:applicationIT}
We consider the simplified version of the
equi-energy sampler~\citep{kou:zhou:wong:2006} introduced in
\cite{andrieu:jasra:doucet:delmoral:2008}. This version is referred to as the
Interacting-tempering (IT) sampler.  Recently, convergence of the marginals and
strong law of large numbers results have been established under general
conditions~(see \cite{fort:moulines:priouret:2010}). In this section, we derive
a CLT under similar assumptions.

Let $\{\pi^{\beta_k}, k \in \{1, \cdots, K\} \}$ be a sequence
  of tempered densities on $\Xset$, where $0<\beta_1 < \cdots < \beta_K=1$. At
  the first level, a process $\sequence{Y}{k}$ with stationary distribution
  proportional to $\pi^{\beta_1}$ is run.  At the second level, a process $\sequence{X}{k}$ with stationary distribution
  proportional to $\pi^{\beta_2}$ is constructed: at each iteration the next value is obtained from a Markov kernel depending on
  the occupation measure of the chain $\sequence{Y}{k}$ up to the current time-step. This $2$-stages mechanism is then repeated to design a process targeting
  $\pi^{\beta_k}$ by using the occupation measure of the process targeting $\pi^{\beta_{k-1}}$.

For ease of exposition, it is assumed that $(\Xset,\Xsigma)$ is a Polish space equipped with its Borel $\sigma$-field,
and the densities are w.r.t. some $\sigma$-finite measure on $(\Xset,\Xsigma)$.
We address the case $K=2$ and discuss below possible  extensions to the case $K>2$.

We start with a description of the IT (case $K=2$).  Denote by $\Tset$ the
set of the probability measures on $(\Xset, \Xsigma)$ equipped with the Borel
sigma-field $\Tsigma$ associated to the topology of weak convergence.  Let $P$ be a transition
kernel on $(\Xset, \Xsigma)$ with unique invariant distribution $\pi$
(typically, $P$ is chosen to be a Metropolis-Hastings kernel). Denote by
$\epsilon \in (0,1)$ the probability of interaction.  Let $\sequence{Y}{k}$
be a discrete-time (possibly non-stationary) process  and denote by $\t_n$ the empirical probability measure:
\begin{equation}
  \label{eq:IT:thetan}
  \t_n \eqdef \frac{1}{n} \sum_{k=1}^n \delta_{Y_k} \eqsp.
\end{equation}
Choose $X_0 \sim \nu$.  At the
$n$-th iteration of the algorithm, two actions may be taken:
\begin{enumerate}
\item with probability $(1-\epsilon)$, the state $X_{n+1}$ is sampled
  from the Markov kernel $P(X_n,\cdot)$,
\item with probability $\epsilon$, a tentative state $Z_{n+1}$ is drawn
  uniformly from the past of the auxiliary process $\{Y_k, k \leq n\}$.  This
  move is accepted with probability $\ratio(X_n,Z_{n+1})$, where the acceptance
  ratio $\ratio$ is given by
\begin{equation}
\label{eq:definition-alpha}
\ratio(x,z) \eqdef 1 \wedge \frac{\pi(z) \pi^{1-\beta}(x)}{\pi^{1-\beta}(z) \pi(x)} = 1 \wedge \frac{\pi^{\beta}(z)}{\pi^{\beta}(x)} \eqsp.
\end{equation}
\end{enumerate}
Define the family of Markov transition kernels $\{P_\t, \t \in \Tset \}$ by
\begin{multline}
\label{eq:definition-Pt-EE}
P_\t(x,A) \eqdef  (1-\epsilon) P(x,A) \\
+ \epsilon \left( \int_A \ratio(x,y) \t(\rmd y) + \1_A(x) \int \left\{ 1-
    \ratio(x,y) \right\} \t(\rmd y) \right) \eqsp.
\end{multline}
Then, the above algorithmic description implies that the bivariate process
$\sequencetwo{X}{\t}{n}$ is such that for any bounded function $h$ on
$\Xset^{n+1}$
\[
\PE\left[h(X_{0:n}) \vert \t_{0:n} \right] = \int \nu(\rmd x_0) P_{\t_0}(x_0,
\rmd x_1) \cdots P_{\t_{n-1}}(x_{n-1}, \rmd x_n) \, h(x_{0:n}) \eqsp.
\]

We apply the results of Section~\ref{sec:mainresult} in order to prove that the
IT process $\sequence{X}{k}$ satisfies a CLT.  To that
goal, it is assumed that the target density $\pi$ and the transition kernel $P$
satisfy the following conditions:
\begin{hypEE}
\label{EES:Pi}
\label{E:BoundPi} $\pi$ is a continuous positive density  on $\Xset $
and $\supnorm{\pi} < +\infty$.
\end{hypEE}
\begin{hypEE}
\label{E:KernelP}
\begin{enumerate}[(a)]
\item \label{E:Irred} $P$ is a phi-irreducible aperiodic Feller transition
  kernel on $(\Xset, \Xsigma)$ such that $\pi P = \pi$.
\item \label{E:Drift} There exist $\tau \in (0,1)$, $\lambda \in (0,1)$ and
  $b<+\infty$  such that
\begin{equation}
\label{eq:definitionW}
P V \leq \lambda V + b  \quad \text{with} \quad
V(x) \eqdef \left( \pi(x) / \supnorm{\pi} \right)^{-\tau} \eqsp.
\end{equation}
\item \label{E:Minor} For any $p \in (0,\supnorm{\pi})$, the sets $\{\pi \geq p \}$ are
  $1$-small (w.r.t. the transition kernel $P$).
\item \label{E:equicontinuite} For any $\gamma \in (0,1/2)$ and any
  equicontinuous set of functions $\F \subseteq \L_{V^\gamma}$, the set of functions $\{P h:
  h \in \F, \fnorm{h}{V^\gamma} \leq 1 \}$ is equicontinuous.
\end{enumerate}
\end{hypEE}
From the expression of the acceptance ratio $r$ (see
Eq.~(\ref{eq:definition-alpha})) and the assumption
I\ref{E:KernelP}-\ref{E:Irred}, it holds
\[
\pi P_{\t_\star} = \pi \eqsp, \] where $\t_\star \propto \pi^{1-\beta}$.
Therefore, when $\t_n$ converges to $\t_\star$, it is expected that $\sequence{X}{k}$ behaves asymptotically as $\pi$;
see~\cite{fort:moulines:priouret:2010}.

Drift conditions for the symmetric random walk Metropolis (SRWM) algorithm are discussed in
\cite{roberts:tweedie:1996}, \cite{jarner:hansen:2000} and
\cite{saksman:vihola:2010}.  Under conditions which imply that the target
density $\pi$ is super-exponential in the tails and have regular contours,
\cite{jarner:hansen:2000} and \cite{saksman:vihola:2010} show that any
functions proportional to $\pi^{-s}$ with $s \in (0,1)$ satisfies a
Foster-Lyapunov drift inequality \cite[Theorems 4.1 and
4.3]{jarner:hansen:2000}.  Under this condition, I\ref{E:KernelP}-\ref{E:Drift}
is satisfied with any $\tau$ in the interval $(0, 1)$. Assumption
I\ref{E:KernelP}-\ref{E:equicontinuite} holds for the SRWM kernel under weak
conditions on the symmetric proposal distribution as shown by the following
lemma. The proof is in section~\ref{sec:proof:lem:equicontinuity:P}.

\begin{lemma}
  \label{lem:equicontinuity:P}
  Assume I\ref{EES:Pi}. Let $P$ be a Metropolis kernel with invariant
  distribution $\pi$ and a symmetric proposal distribution $q: \Xset \times
  \Xset \to \Rset^+$ such that $\sup_{(x,y) \in \Xset^2} q(x,y) < +\infty$ and
  the function $x \mapsto q(x,\cdot)$ is continuous from $(\Xset,|\cdot|)$ to
  the set of probability densities equipped with the total variation norm.
  Then $P$ satisfies I\ref{E:KernelP}-\ref{E:equicontinuite} with any function
  $V \propto \pi^{-\tau}$, $\tau \in \coint{0, 1}$, such that
  $\pi(V) < +\infty$.
\end{lemma}

For a measurable function $f: \Xset \to \Rset$ such that $\t_\star(|f|) <
+\infty$, define the following sequence of random processes on $[0,1]$:
\begin{equation}
\label{eq:definition-S_n}
t \mapsto S_n(f;t)= n^{-1/2} \sum_{j=1}^{\lfloor nt \rfloor} \left\{ f(Y_j) - \t_\star(f) \right\} \eqsp.
\end{equation}
It is assumed that the auxiliary process $\Yproc$ converges to the probability
distribution $\theta_\star$ in the following sense:
\begin{hypEE}
\label{E:ProcY}
\begin{enumerate}[(a)]
\item \label{E:ProcY:item1} $\theta_\star(V) < +\infty$ and $\sup_n
  \PE\left[V(Y_n) \right]< + \infty$.
\item \label{E:ProcY:item2} There exists a space $\mathcal{N}$ of real-valued
  measurable functions defined on $\Xset$ such that $V \in \mathcal{N}$ and for
  any function $f \in \mathcal{N}$, $\t_n(f) \aslim \t_\star(f)$.
\item \label{E:ProcY:item3} For any function $f \in \mathcal{N}$, the sequence
  of processes $( S_n(f,t), n \geq 1, t \in \ccint{0,1})$ converges in distribution to $ (\tilde{\gamma}(f)
  B(t), t \in \ccint{0,1}) $, where $\tilde{\gamma}(f)$ is a non-negative constant and $
  \left( B(t): t \in \ccint{0, 1} \right)$ is a standard Brownian motion.
\item \label{E:ProcY:item4} For any $\alpha \in (0,1/2)$, there exist constants
  $\varrho_0$ and $\varrho_1$ such that, for any integers $n,k \geq 1$, for any
  measurable function $h: \Xset^k \to \Rset$ satisfying $|h(y_1,\dots,y_k)|\leq
  \sum_{j=1}^k V^\alpha(y_j)$,
  \[
  \PE{\left( \idotsint \prod_{j=1}^k \left[\t_n(\rmd y_j) - \t_\star(\rmd
        y_j)\right] h(y_1,\dots,y_k) \right)^2 } \leq A_k \, n^{-k} \eqsp,
  \]
with $\limsup_k \ln A_k /(k \ln k) < \infty$.
\end{enumerate}
\end{hypEE}
  I\ref{E:ProcY} is satisfied when
$\sequence{Y}{k}$ is i.i.d. with distribution $\t_\star$ such that $\t_\star(V)
< +\infty$. In that case, I\ref{E:ProcY}-\ref{E:ProcY:item2} to
I\ref{E:ProcY}-\ref{E:ProcY:item3} hold for any measurable function $f$ such
that $\t_\star(|f|^2) < +\infty$.  I\ref{E:ProcY}-\ref{E:ProcY:item4} is
satisfied using \cite[Lemma~A, pp.  190]{serfling:1980}.

I\ref{E:ProcY} is also satisfied when $\sequence{Y}{k}$ is an asymptotically
stationary Markov chain with transition kernel $Q$. In that case,
I\ref{E:ProcY}-\ref{E:ProcY:item1} to I\ref{E:ProcY}-\ref{E:ProcY:item3} are
satisfied for any measurable function $f$ such that $\t_\star \left( |f \,
  [(I-Q)^{-1} f]| \right) < +\infty$ (see e.g. \cite[Chapter
17]{meyn:tweedie:2009}). Condition I\ref{E:ProcY}-\ref{E:ProcY:item4} for a
(non-stationary) geometrically ergodic Markov chain is established in the
supplementary paper~\citep{fort:moulines:priouret:vdk:2011-supplement}.
  
The following proposition shows that under I\ref{EES:Pi} and I\ref{E:KernelP},
condition \textbf{A\ref{hyp:geometric-ergodicity}} holds with the drift function $V$
given by A\ref{E:KernelP}-\ref{E:Drift}. It also provides a control of the
ergodicity constants $C_\t, \rho_\t$ in Lemma~\ref{lem:BoundCandRho}.  The
proof is a direct consequence of~\cite[Proposition 3.1,
Corollary~3.2]{fort:moulines:priouret:2010}, Lemmas~\ref{lem:BoundCandRho} and
\ref{lem:explicit:control}, and is omitted.
\begin{prop}
  \label{prop:A1:IT}
  Assume I\ref{EES:Pi} and
  I\ref{E:KernelP}\ref{E:Irred}-\ref{E:Drift}-\ref{E:Minor}. For any $\t \in
  \Tset$, $P_\t$ is phi-irreducible, aperiodic. In addition, there exist
  $\tilde \lambda \in (0,1)$ and $\tilde b < +\infty$ such that, for any $\t
  \in \Tset$,
\begin{equation}\label{eq:prop:A1:IT:drift}
P_\t V(x) \leq \tilde \lambda  V(x) + \tilde b \, \t(V) \eqsp, \quad \text{for all $x \in \Xset$.}
\end{equation}
The property P[$\alpha$] holds for any $\alpha \in (0,1/2)$, and there exists
$C$ such that for any $\t \in \Tset$, $L_\t \leq C \t(V)$.

Assume in addition I\ref{E:ProcY}\ref{E:ProcY:item1} and $\PE[V(X_0)]< +
\infty$. Then, $\sup_{n \geq 0} \PE\left[V(X_n) \right] < +\infty$.
\end{prop}

The next step is to check assumptions A\ref{hyp:cvgProba:series} and
\textbf{A\ref{hyp:existence:variance}}.

\begin{prop}
\label{prop:A3:A4:ICMCM}
Assume I\ref{EES:Pi}, I\ref{E:KernelP},
I\ref{E:ProcY}\ref{E:ProcY:item1}-\ref{E:ProcY:item2} and $\PE[V(X_0)]< +
\infty$.  For any $\alpha \in (0,1/2)$, set $\Lsub_{V^\alpha}$ be the set of
continuous functions belonging to $\L_{V^\alpha} \cap \mathcal{N}$. Then, for
any $\alpha \in (0,1/2)$, the conditions A\ref{hyp:cvgProba:series} and
\textbf{A\ref{hyp:existence:variance}} hold with
\begin{equation}
  \label{eq:IT:variance1}
  \sigma^2(f) \eqdef \int \pi_{\t_\star}(\rmd x)  F_{\t_\star}(x) \eqsp,
\end{equation}
where $F_\t$ is given by (\ref{eq:definition-gtheta}).
\end{prop}
The proof is postponed to Appendix~\ref{sec:proof:prop:A3:A4:ICMCM}. We can now
apply Theorem~\ref{theo:TCL} and prove a CLT for the $2$-levels IT.
\begin{theo}
\label{theo:IT}
  Assume I\ref{EES:Pi}, I\ref{E:KernelP}, I\ref{E:ProcY} and $\PE[V(X_0)]< +
  \infty$. Then, for any $\alpha \in (0,1/2)$ and any continuous function $f
  \in \L_{V^\alpha} \cap \mathcal{N}$,
\[
\frac{1}{\sqrt{n}} \sum_{k=1}^n \left( f(X_k) - \pi_{\t_\star}(f) \right) \dlim
\mathcal{N}(0, \sigma^2(f) + 2 \tilde \gamma^2(f)) \eqsp,
\]
where $\sigma^2(f)$ and $\tilde \gamma^2(f)$ are given by
(\ref{eq:IT:variance1}) and I\ref{E:ProcY}-\ref{E:ProcY:item3}.
\end{theo}
The proof is postponed to Appendix~\ref{sec:proof:theo:IT}.

The above discussion could be repeated in order to prove by induction a CLT for
the $K$-level IT when $K>2$ (see \cite{fort:moulines:priouret:2010} for a
similar approach in the proof of the ergodicity and the LLN for IT).
Nevertheless, the main difficulty is to iterate the control of the $L^2$-moment
for the $V$-statistics (see I3-\ref{E:ProcY:item4}) when $\sequence{Y}{k}$ is
not a Markov Chain or, more generally, a process satisfying some mixing
conditions.  A similar difficulty has been reported in
\cite{andrieu:jasra:doucet:delmoral:2008}.

\section{Proofs}
\label{sec:proofs}
Denote by $D_V(\t,\t')$ the $V$-distance of the kernels $P_\t$ and $P_{\t'}$:
\begin{equation}
\label{eq:defi:DVnorm}
D_{V}(\t, \t') \eqdef \Vnorm{P_\t - P_{\t'}}{V} \eqsp.
\end{equation}
Note that under \textbf{A\ref{hyp:geometric-ergodicity}}, for any $\alpha \in \ocint{0,1}$,
any $f \in \L_{V^\alpha}$ and any $\t \in \Tset$,
 \begin{equation}
    \label{eq:Controle:FctPoisson}
    \fnorm{\operpoisson_\t f}{V^\alpha} \leq \fnorm{f}{V^\alpha} \ L_\t^2
  \end{equation}
  where $ L_\t$ is defined by (\ref{eq:DefinitionLtheta}).
\subsection{Proofs of the results in Section~\ref{sec:mainresult}}
\subsubsection{Proof of Theorem~\ref{theo:CLT:martingale}}
\label{sec:proof:theo:CLT:martingale}
Let $f \in \Lsub_{V^\alpha}$.  Eq. (\ref{eq:PoissonEquation}) yields to
$S_n^{(1)}(f) = \Xi_n(f) +R^{(1)}_n(f) + R_n^{(2)}(f)$ with
\begin{align*}
  \Xi_n(f) &\eqdef  \frac{1}{\sqrt{n}} \sum_{k=1}^n \{ \operpoisson_{\theta_{k-1}}f(X_k) -
  P_{\theta_{k-1}} \operpoisson_{\theta_{k-1}}f(X_{k-1})\} \eqsp, \\
  R^{(1)}_n(f) &\eqdef n^{-1/2} \sum_{k=1}^n \{ P_{\theta_{k}} \operpoisson_{\theta_{k}}f(X_{k}) -  P_{\theta_{k-1}} \operpoisson_{\theta_{k-1}}f(X_{k}) \} \eqsp, \\
  R^{(2)}_n(f) &\eqdef n^{-1/2} P_{\theta_0} \operpoisson_{\theta_0}f(X_0) -
  n^{-1/2} P_{\theta_n} \operpoisson_{\theta_n}f(X_n)\eqsp.
\end{align*}
We first show that the two remainders terms $R_n^{(1)}(f)$ and $R_n^{(2)}(f)$
converge to zero in probability.  We have
\[
\left| P_{\theta} \operpoisson_{\theta}f(x) -P_{\theta'}
  \operpoisson_{\theta'}f(x) \right| \leq \fnorm{P_{\theta}
  \operpoisson_{\theta}f(x) -P_{\theta'}
  \operpoisson_{\theta'}f(x)}{V^{\alpha}} V^{\alpha}(x) \eqsp.
\]
Assumption A\ref{hyp:cvgProba:series} implies that $R_n^{(1)}(f)$ converges to
zero in probability.  The drift inequality \textbf{A\ref{hyp:geometric-ergodicity}}
combined with the Jensen's inequality imply $P_\t V^{\alpha} \leq
\lambda_\t^\alpha V^{\alpha} + b^\alpha_\t$.  By (\ref{eq:Controle:FctPoisson})
and this inequality,
\[
|P_\t \operpoisson_{\theta}f(x)| \leq \fnorm{f}{V^{\alpha}} \, L_\t^2 \, P_\t
V^{\alpha}(x) \leq \fnorm{f}{V^{\alpha}} \, L_\t^2 \, (V^{\alpha}(x) +b^\alpha_\t) \eqsp.
\]
Then, $P_{\theta_0} \operpoisson_{\theta_0}f(X_0)$ is finite w.p.1. and
$n^{-1/2} P_{\theta_0} \operpoisson_{\theta_0}f(X_0) \aslim 0$. By
\textbf{A\ref{hyp:cvgProba:series}-\ref{hyp:cvgProba:series:2}} and
(\ref{eq:Controle:FctPoisson}), $n^{-1/2} P_{\theta_n}
\operpoisson_{\theta_n}f(X_n) \plim 0$. Hence, $R_n^{(2)}(f) \plim 0$.

We now consider $\Xi_n(f)$. Set $D_k(f) \eqdef \operpoisson_{\t_{k-1}}f(X_k) -
P_{\t_{k-1}} \operpoisson_{\t_{k-1}}f(X_{k-1})$.  Observe that under
\textbf{A\ref{hyp:DefinitionProc:X:Theta}}, $D_k(f)$ is a martingale-increment w.r.t.
the filtration $\{\F_k, k\geq 0\}$. The limiting distribution for $\Xi_n(f)$
follows from martingale CLT (see \eg\ \cite[Corollary 3.1.]{hall:heyde:1980}).
We check the conditional Lindeberg condition. Let $\epsilon >0$. Under
\textbf{A\ref{hyp:geometric-ergodicity}}, we have by (\ref{eq:Controle:FctPoisson})
\[
D_k(f) \leq \fnorm{f}{V^\alpha} \ \left| L_{\theta_{k-1}}^2 \, \left\{ V^\alpha(X_k) +
    P_{\theta_{k-1}}V^\alpha(X_{k-1}) \right\} \right|\eqsp.
\]
Set $\tau \eqdef 1/\alpha-2>0$.
\begin{align*}
&  \frac{1}{n}\sum_{k=1}^n \CPE{D_k^2(f) \un_{|D_k(f)| \geq \epsilon \sqrt{n}}}{\mathcal{F}_{k-1}} \leq \left(\frac{1}{\epsilon \sqrt{n}} \right)^\tau \frac{1}{n}\sum_{k=1}^n
  \CPE{D_k^{2 +\tau}(f)}{\mathcal{F}_{k-1}} \\
&  \leq \fnorm{f}{V^\alpha}^{2 +\tau} \ \left(\frac{1}{\epsilon \sqrt{n}} \right)^\tau
  \frac{1}{n}\sum_{k=1}^n \CPE{L_{\theta_{k-1}}^{2(2+\tau)} \, \left\{
      V^\alpha(X_k) + P_{\theta_{k-1}}V^\alpha(X_{k-1}) \right\}^{2 +\tau}}{\mathcal{F}_{k-1}} \\
&  \leq 2^{2+\tau} \fnorm{f}{V^\alpha}^{2 +\tau} \ \left(\frac{1}{\epsilon \sqrt{n}}
  \right)^\tau \frac{1}{n}\sum_{k=0}^{n-1}  L_{\theta_{k}}^{2(2+\tau)} \,
  P_{\theta_{k}}V(X_{k}) \eqsp.
\end{align*}
Under \textbf{A\ref{hyp:cvgProba:series}-\ref{hyp:cvgProba:series:2}}, the RHS converges
to zero in probability thus concluding the proof of the conditional Lindeberg
condition. For the limiting variance condition, observe that
\[
\frac{1}{n} \sum_{k=1}^n \CPE{D_k^2(f)}{\F_{k-1}}
  = \frac{1}{n} \sum_{k=0}^{n-1} F_{\t_k}(X_k) \eqsp,
\]
where $F_\t$ is given by \eqref{eq:definition-gtheta} and, under
\textbf{A\ref{hyp:existence:variance}}, $ n^{-1} \sum_{k=1}^n \CPE{D_k^2}{\F_{k-1}}
\plim \sigma^2(f)$.  This concludes the proof.

\subsubsection{Proof of Theorem~\ref{theo:TCL}}
\label{sec:proof:theo:TCL}
We start by establishing a joint CLT for $(S_n^{(1)}(f),S_n^{(2)}(f))$, where
$S_n^{(1)}(f)$ and $S_n^{(2)}(f)$ are defined in \eqref{eq:definition-S_1} and
\eqref{eq:definition-S_2}, respectively. Similar to the proof of
Theorem~\ref{theo:CLT:martingale}, we write $S_n^{(1)}(f) = \Xi_n(f)
+R^{(1)}_n(f) + R_n^{(2)}(f)$ and prove that $R^{(1)}_n(f) + R_n^{(2)}(f) \plim
0$. We thus consider the convergence of $\Xi_n(f) + S_n^{(2)}(f)$. Set $\F_n^\t
\eqdef \sigma(\t_k, k\leq n)$. Under \textbf{A\ref{hyp:X_conditionellement_markov}},
\[
\PE\left[\rme^{\rmi (u_1 \Xi_n(f) + u_2 S_n^{(2)}(f))}\right]= \PE\left[
  \CPE{\rme^{\rmi u_1 \Xi_n(f)}}{\F_n^\t} \rme^{\rmi u_2 S_n^{(2)}(f)} \right]
\eqsp.
\]
Applying the conditional CLT \cite[Theorem A.3.]{douc:moulines:2008} with the filtration
$\F_{n,k} \eqdef \sigma(Y_{1}, \cdots, Y_n, X_1, \cdots, X_k)$, yields to:
\begin{equation}
  \label{eq:tool:douc:moulines}
  \lim_{n \to \infty} \CPE{\rme^{\rmi u_1 \Xi_n(f)}}{\F_n^\t} \plim \rme^{-u_1^2
  \sigma^2(f)/2} \eqsp;
\end{equation}
observe that under \textbf{A\ref{hyp:X_conditionellement_markov}}, the conditions (31)
and (32) in \cite{douc:moulines:2008} can be proved following the same lines as
in the proof of Theorem~\ref{theo:CLT:martingale}; details are omitted.
Therefore,
\begin{multline*}
  \PE\left[\rme^{\rmi (u_1 \Xi_n(f) + u_2 S_n^{(2)}(f))}\right]=
  \PE\left[ \left( \CPE{\rme^{\rmi u_1 \Xi_n(f)}}{\F_n^\t} - \rme^{-u_1^2 \sigma^2(f)/2} \right) \rme^{i u_2 S_n^{(2)}(f)}  \right] \\
  + \rme^{-u_1 \sigma^2(f)/2} \PE\left[ \rme^{\rmi u_2 S_n^{(2)}(f)} \right]
  \eqsp.
\end{multline*}
By (\ref{eq:tool:douc:moulines}), the first term in the RHS of the previous
equation converges to zero.  Under \textbf{A\ref{hyp:variance:TCLsurPi}}, $\lim_{n \to
  \infty} \PE\left[ \rme^{\rmi u_2 S_n^{(2)}(f)} \right] = \rme^{-u_2^2
  \gamma^2(f)/2}$ and this concludes the proof.

\subsection{Proofs of Section~\ref{sec:applicationIT}}
\label{sec:proof:IT}
\subsubsection{Proof of Lemma~\ref{lem:equicontinuity:P}}
\label{sec:proof:lem:equicontinuity:P}
Let $\gamma \in (0,1/2)$ and $\F$ be an equicontinuous set of functions in
$\L_{V^\gamma}$. Let $h \in \F$, $\fnorm{h}{V^\gamma} \leq 1$. By construction,
the transition kernel of a symmetric random walk Metropolis with proposal
transition density $q(x,\cdot)$ and target density $\pi$ may be expressed as
\[
Ph(x)= \int \ratio(x,y) h(y) q(x,y) \, \rmd y + h(x) \int \left\{ 1 -
  \ratio(x,y) \right\} q(x,y) \rmd y \eqsp,
\]
where $\ratio(x,y) \eqdef 1 \wedge (\pi(y)/\pi(x))$ is the acceptance ratio. Therefore, the difference
$Ph(x) - Ph(x')$ may be bounded by
  \begin{align*}
    & \left| P h(x) - P h(x') \right| \leq  2  \left|h(x) - h(x') \right|  \\
    & \phantom{| P h(x) - } + \int \left|h(y) - h(x') \right| \ \left|
      \ratio(x,y) -
      \ratio(x',y) \right| q(x,y) \rmd y  \\
    & \phantom{| P h(x) - } +\left| \int \left(h(y) - h(x') \right) \
      \ratio(x',y) \left( q(x,y) - q(x',y) \right) \rmd y \right| \eqsp.
  \end{align*}
Since   $|\ratio(x,y) - \ratio(x',y)| \leq \pi(y) | \pi^{-1}(x) - \pi^{-1}(x') |$,
\begin{align*}
&\int \left|h(y) - h(x') \right| \ \left| \ratio(x,y) - \ratio(x',y) \right| q(x,y) \rmd y \\
&\phantom{\int \|h(y) }   \leq \left| \pi^{-1}(x) - \pi^{-1}(x') \right| \ \int \left|h(y) - h(x') \right| \pi(y) \ q(x,y) \ \rmd y \\
&\phantom{\int \|h(y) } \leq  \left( \sup_{(x,y) \in \Xset^2} q(x,y) \right) \ \left| \pi^{-1}(x) - \pi^{-1}(x') \right| \ \left( \pi(V^\gamma) + V^\gamma(x') \right) \eqsp.
\end{align*}
In addition,
\begin{align*}
& \left| \int \left(h(y) - h(x') \right) \ \ratio(x',y) \left( q(x,y) - q(x',y)
    \right) \rmd y \right| \\
&  = \left| \int_{\{y: \pi(y) \leq \pi(x') \}} \left(h(y) - h(x') \right) \
    \frac{\pi(y)}{\pi(x')} \left( q(x,y) - q(x',y)
    \right) \rmd y \right| \\
&  + \left| \int_{\{y: \pi(y) > \pi(x') \}} \left(h(y) - h(x') \right) \ \left(  q(x,y) - q(x',y) \right) \rmd y \right| \\
& \leq 4 \ \pi^{-1}(x') \ \tvnorm{q(x,\cdot) - q(x',\cdot)} \ \sup_{y \in \Xset} \left| h(y)\ \pi(y)\right| \eqsp.
\end{align*}
Since $V \propto \pi^{-\tau}$ and $\tau \in (0,1)$, $\sup_{\Xset} |h| \pi \leq
1$ under I\ref{EES:Pi}.  Therefore,  there exists a constant $C$ such that for
any $h \in \{h \in \F, \fnorm{h}{V^\gamma} \leq 1 \}$ and any $x,x' \in \Xset$,
\begin{multline*}
  \left| P h(x) - P h(x') \right| \leq 2 \left|h(x) - h(x') \right| \\
  + C \left(\left| \pi^{-1}(x) - \pi^{-1}(x') \right| +\tvnorm{q(x,\cdot) -
      q(x',\cdot)} \right) \left( V^\gamma(x') + \pi^{-1}(x') \right) \eqsp,
\end{multline*}
thus concluding the proof.

\subsubsection{Proof of Proposition~\ref{prop:A3:A4:ICMCM}}
\label{sec:proof:prop:A3:A4:ICMCM}
The proof is prefaced by several lemmas.
\begin{lemma}
\label{lem:iMCMC:majoDV}
Let $P_\t$ be the transition kernel given by \eqref{eq:definition-Pt-EE}. For
any $(\t,\t') \in \Tset^2$ and any $\alpha \in \ocint{0,1}$, we
have $D_{V^\alpha}(\t,\t') \leq 2 \Vnorm{\t-\t'}{V^\alpha}$.  For any positive
integer $n$,
\begin{equation}
\label{eq:boundDV}
D_{V^\alpha}(\t_n,\t_{n-1}) \leq \frac{2}{n} \t_{n-1}(V^\alpha) + \frac{2}{n} V^\alpha(Y_n) \eqsp.
\end{equation}
\end{lemma}
\begin{proof}
  For any $f \in \L_{V^\alpha}$ and any $x \in \Xset$,
\begin{multline*}
  \left| P_\t(x,f) - P_{\t'}(x,f) \right| = \epsilon \left| \int \ratio(x,y) \left[ f(y) - f(x) \right] \left[ \t(\rmd y) - \t'(\rmd y) \right] \right| \\
  \leq \epsilon \Vnorm{\t - \t'}{V^\alpha} \fnorm{\ratio(x,\cdot) \left[
      f(\cdot) - f(x) \right]}{V^\alpha} \leq 2 \Vnorm{\t - \t'}{V^\alpha}
  \fnorm{f}{V^\alpha} \eqsp,
\end{multline*}
which proves the first assertion. Inequality  \eqref{eq:boundDV} follows from
(\ref{eq:IT:thetan}) and the obvious identity
\[
\t_n(f) - \t_{n-1}(f) = \frac{-1}{n(n-1)} \sum_{k=1}^{n-1} f(Y_k) + \frac{1}{n} f(Y_n) = \frac{1}{n}\left[ f(Y_n)- \t_{n-1}(f) \right]\eqsp.
\]
\end{proof}

\begin{lemma}
\label{lem:bienpratique}
Let $\alpha \in (0,1)$.  Assume I\ref{EES:Pi},
I\ref{E:KernelP}\ref{E:Irred}-\ref{E:Drift}-\ref{E:Minor},
I\ref{E:ProcY}\ref{E:ProcY:item1}-\ref{E:ProcY:item2}, and $\PE[V(X_0)]< +
\infty$. Then for any $\gamma, \gamma' \in (0,1)$ and any $\delta > \gamma$,
\[
n^{-\delta} \sum_{k=1}^n D_{V^\gamma}(\t_k, \t_{k-1}) V^{\gamma'}(X_k) \plim 0 \eqsp.
\]
\end{lemma}
\begin{proof}
  By Lemma~\ref{lem:iMCMC:majoDV}, we have
\[
n^{-\delta} \sum_{k=1}^n D_{V^\gamma}(\t_k, \t_{k-1}) V^{\gamma'}(X_k) \leq 2
n^{-\delta} \sum_{k=1}^n \frac{1}{k} \left\{ \t_{k-1}(V^\gamma) + V^\gamma(Y_k)
\right\} V^{\gamma'}(X_k) \eqsp.
\]
By I\ref{E:ProcY}-\ref{E:ProcY:item2}, $\t_k(V) \aslim \t_\star(V)$ thus
implying that $\limsup_k \{ \t_k(V^\gamma) + k^{-\gamma} V^\gamma(Y_k)\} <
\infty$, $\PP$-\as .  Therefore, the result holds if
\[
\limsup_{n \to \infty} n^{-\delta} \sum_{k =1}^n k^{\gamma-1} \PE[V^{\gamma'}(X_k)] = 0
\eqsp.
\]
Under the stated assumptions, Proposition~\ref{prop:A1:IT} implies that $\sup_k \PE\left[V(X_k) \right]< +\infty$ and this concludes the proof.
\end{proof}

\begin{lemma}\label{lem:equicontinuity:Ptheta}
  For any $\t \in \Tset$, any measurable function $f : \Xset \to \Rset$ in
  $\L_{V^\alpha}$ and any $x,x' \in \Xset$ such that $\pi(x) \leq \pi(x')$
  \begin{multline*}
    \left| P_\t f(x) - P_\t f(x') \right| \leq \left| Pf(x) - Pf(x') \right| +
    \left|f(x) - f(x') \right|   \\
    + \sup_\Xset \pi \ \fnorm{f}{V^\alpha} \ \left| \pi^{-\beta}(x) -
      \pi^{-\beta}(x') \right| \ \left(V^\alpha(x') + \t(V^\alpha) \right)
    \eqsp.
  \end{multline*}
\end{lemma}
\begin{proof}
The proof is adapted from \cite[Lemma~5.1.]{fort:moulines:priouret:2010}; it is  omitted for brevity.
\end{proof}

\begin{proof}[Proof of Proposition~\ref{prop:A3:A4:ICMCM}]
  Let $\alpha \in (0,1/2)$. By Proposition~\ref{prop:A1:IT},
  \textbf{A\ref{hyp:geometric-ergodicity}} and P[$\alpha$] hold.  By
  I\ref{E:ProcY}-\ref{E:ProcY:item2},
  \begin{equation}
    \label{eq:IT:LthetaCvg}
    \limsup_n L_{\t_n} < +\infty \eqsp, \ \ \PP-\as
  \end{equation}
  where $L_\t$ is given by (\ref{eq:DefinitionLtheta}) with $C_\t, \rho_\t$
  defined by P[$\alpha$].

  We first check \textbf{A\ref{hyp:cvgProba:series}-\ref{hyp:cvgProba:series:1}}. Let $f
  \in \mathcal{N} \cap \L_{V^\alpha}$. By
  Lemma~\ref{lem:regularity-in-theta:poisson},
\[
\fnorm{P_{\t_k} \operpoisson_{\t_k} f - P_{\t_{k-1}} \operpoisson_{\t_{k-1}} f
}{V^\alpha} \leq 5 \ \left(L_{\t_k} \vee L_{\t_{k-1}} \right)^6
\pi_{\t_k}(V^\alpha) D_{V^\alpha}(\t_k,\t_{k-1}) \ \fnorm{f}{V^\alpha}\eqsp.
\]
By Lemma~\ref{lem:BoundCandRho}, Proposition~\ref{prop:A1:IT} and
Assumptions I\ref{EES:Pi}, I\ref{E:KernelP} and
I\ref{E:ProcY}-\ref{E:ProcY:item2},
\begin{equation}
\label{eq:limsuppitheta}
\limsup_{n \to \infty} \pi_{\t_n} (V) \leq \tilde{b} \; (1-\tilde \lambda)^{-1} \limsup_{n \to \infty} \t_n(V) < \infty \eqsp,\ \PP-\as\ .
\end{equation}
Therefore, by (\ref{eq:IT:LthetaCvg}) and (\ref{eq:limsuppitheta}), it
suffices to prove that
\[
n^{-1/2} \sum_{k=1}^n D_{V^\alpha}(\t_k,\t_{k-1}) V^\alpha(X_k) \plim 0 \eqsp,
\]
which follows from Lemma~\ref{lem:bienpratique}. We now check
\textbf{A\ref{hyp:cvgProba:series}-\ref{hyp:cvgProba:series:2}}.  By
Proposition~\ref{prop:A1:IT}, it holds
\[
n^{-1/(2\alpha)} \sum_{k=1}^n L_{\t_k}^{2/\alpha} P_{\t_k} V(X_k) \leq  n^{-1/(2\alpha)} \sum_{k=1}^n L_{\t_k}^{2/\alpha} \left[ V(X_k)  + \tilde b \t_k(V) \right] \eqsp.
\]
Under the stated assumptions, $\limsup_n \left[\t_n(V) + L_{\t_n} \right]<
+\infty$ w.p.1. and by Proposition~\ref{prop:A1:IT}, $\sup_k
\PE\left[V(X_k) \right] < +\infty$. Since $2 \alpha <1$, this concludes the
proof.

The proof of \textbf{A\ref{hyp:existence:variance}} is in two steps: it is first proved
that
\begin{equation}
  \label{eq:prop:A3:A4:IT:tool1}
 \frac{1}{n} \sum_{k=0}^{n-1} F_{\t_k}(X_k)- \frac{1}{n} \sum_{k=0}^{n-1} \int \pi_{\t_k}(\rmd x)  F_{\t_k}(x) \plim 0 \eqsp,
\end{equation}
and then it is established that
\begin{equation}
  \label{eq:prop:A3:A4:IT:tool2}
\int \pi_{\t_k}(\rmd x) F_{\t_k}(x) \aslim \int \pi_{\t_\star}(\rmd x)
F_{\t_\star}(x) \eqsp.
 \end{equation}
 In order to prove (\ref{eq:prop:A3:A4:IT:tool1}), we check the conditions
 of Theorem~\ref{theo:WeakLLN} in Appendix~\ref{appendix} with $\gamma = 2
 \alpha$.  First observe that $\operpoisson_\t f^2 \in \L_{V^{2 \alpha}}$ (see
 (\ref{eq:Controle:FctPoisson})).  We check
 conditions~\eqref{theo:WeakLLN:item1} to~\eqref{theo:WeakLLN:item6} of
 Theorem~\ref{theo:WeakLLN}.

 \textit{\eqref{theo:WeakLLN:item1}} and \textit{\eqref{theo:WeakLLN:item3}}
 follow from Proposition~\ref{prop:A1:IT} and (\ref{eq:IT:LthetaCvg}).

\textit{\eqref{theo:WeakLLN:item2}} follows from Eq.~\eqref{eq:limsuppitheta}.

\textit{\eqref{theo:WeakLLN:item4}} follows from Lemma~\ref{lem:bienpratique}.

\textit{\eqref{theo:WeakLLN:item5}} under \textbf{A\ref{hyp:geometric-ergodicity}}, we have by
(\ref{eq:Controle:FctPoisson}) and the Jensen's inequality
\[
|F_{\t}(x)| \leq 2 \fnorm{f}{V^\alpha}^2 \,  L_\t^4 \ P_\t V^{2 \alpha}(x) \leq 2 \fnorm{f}{V^\alpha}^2
\, L_\t^4 \, \fnorm{P_\t V^{2 \alpha}}{V^{2 \alpha}} V^{2 \alpha}(x) \eqsp.
\]
Hence, $\fnorm{F_{\t}}{V^{2\alpha}} \leq 2 \fnorm{f}{V^{\alpha}}^2 L_\t^4
\fnorm{P_\t V^{2 \alpha}}{V^{2 \alpha}}$.  By
I\ref{E:ProcY}-\ref{E:ProcY:item2}, the drift inequality
(\ref{eq:prop:A1:IT:drift}) and the Jensen's inequality, $\limsup_n
\fnorm{P_{\t_n} V^{2 \alpha}}{V^{2 \alpha}}< +\infty$ w.p.1.
(\ref{eq:IT:LthetaCvg})  concludes the proof.

\textit{\eqref{theo:WeakLLN:item6}} Set $F_\theta(x) = G_\t(x) - H_{\t}(x)$ where
$G_\t(x) \eqdef P_\t [\operpoisson_\t f]^2(x)$ and $H_\t(x) \eqdef \left(P_\t
  \operpoisson_\t f(x) \right)^2$, with

\[
\fnorm{\operpoisson_\t f}{V^\alpha} + \fnorm{\operpoisson_{\t'} f}{V^\alpha} \leq M_{\t,\t'} \eqdef \left( L_{\t}^2 + L_{\t'}^2 \right) \fnorm{f}{V^\alpha} \eqsp.
\]
By (\ref{eq:IT:LthetaCvg}), $\limsup_{n \to \infty} M_{\t_n,\t_{n-1}} <
\infty$, $\PP$-\as

It holds
\begin{align*}
  &  |G_\t(x) -G_{\t'}(x) |  \\
  &\leq \left|P_\t\left(x,[\operpoisson_\t f]^2 - [\operpoisson_{\t'}
      f]^2\right) \right| + \left| \int \left\{ P_\t(x,\rmd y) - P_{\t'}(x,
      \rmd y)\right\}
    [\operpoisson_{\t'} f]^2(y) \right| \\
  & \leq 2 M_{\t,\t'} \ P_\t\left(x, \left| \operpoisson_\t f -
      \operpoisson_{\t'} f\right| V^\alpha\right) + M_{\t,\t'}^2 \ D_{V^{2 \alpha}}(\t,\t') \ V^{2 \alpha}(x)  \\
  & \leq 2 M_{\t,\t'} \ \fnorm{\operpoisson_\t f - \operpoisson_{\t'}
    f}{V^\alpha} \ \fnorm{P_\t V^{2 \alpha}}{V^{2\alpha}} V^{2\alpha}(x) +
  M_{\t,\t'}^2 \ D_{V^{2 \alpha}}(\t,\t') \ V^{2\alpha}(x) \eqsp.
\end{align*}
By Lemma~\ref{lem:regularity-in-theta:poisson},
$$
\fnorm{f}{V^\alpha}^{-1} \fnorm{\operpoisson_\t f - \operpoisson_{\t'} f}{V^\alpha} \leq
3  D_{V^{\alpha}}(\t,\t') \ \left(L_\t \vee L_{\t'} \right)^6 \pi_\t(V^\alpha) \eqsp.
$$
Since w.p.1.:
\[
\limsup_{n \to \infty} \left\{ \pi_{\t_n}(V) + M_{\t_n,\t_{n-1}} + L_{\t_n} +
  \fnorm{P_{\t_n} V^{2\alpha}}{V^{2\alpha}} \right\}< \infty \eqsp,
\]
it follows that $n^{-1} \sum_{k=1}^n V^{2\alpha}(X_k) \fnorm{G_{\t_k}-
  G_{\t_{k-1}}}{V^{2 \alpha}}$ converges to zero in probability provided that
\[
n^{-1} \sum_{k=1}^n \left[ D_{V^{2\alpha}}(\t_k,\t_{k-1}) +
  D_{V^\alpha}(\t_k,\t_{k-1}) \right] V^{2 \alpha}(X_k) \plim 0 \eqsp,
\]
which follows from Lemma~\ref{lem:bienpratique}. Similarly, it  holds
\begin{multline*}
  |H_{\t}(x) -H_{\t'}(x) | \leq \left| P_\t \operpoisson_\t f(x) - P_{\t'} \operpoisson_{\t'} f(x)
  \right|\left| P_\t \operpoisson_\t f(x) + P_{\t'} \operpoisson_{\t'} f(x) \right| \\
  \leq M_{\t,\t'} \ \left| P_\t \operpoisson_\t f(x) - P_{\t'} \operpoisson_{\t'} f(x) \right| \{
  P_\t V^\alpha(x) + P_{\t'} V^\alpha(x) \} \eqsp.
\end{multline*}
Along the same lines as above, using Lemmas~\ref{lem:bienpratique} and
\ref{lem:regularity-in-theta:poisson}, we prove that $n^{-1} \sum_{k=1}^n
V^{2\alpha}(X_k) \fnorm{H_{\t_k}- H_{\t_{k-1}}}{V^{2 \alpha}}$ converges to $0$
in probability.

We now consider the second step and prove~(\ref{eq:prop:A3:A4:IT:tool2}).
To that goal, we have to strengthen the conditions on $f$ by assuming that $f$
is continuous.  For any $\t \in \Tset$, $ \int \pi_\t(\rmd x) F_\t(x) = \int
\pi_\t(\rmd x) H_\t(x)$ with
\begin{equation}
  \label{eq:fonction:Htheta}
H_\t(x) \eqdef \left(\operpoisson_\t f
\right)^2(x) -  \left( P_\t \operpoisson_\t f \right)^2(x)
\eqsp.
\end{equation}
We prove that there exists $\Omega_\star$ with $\PP(\Omega_\star)=1$
and for any $\omega \in \Omega_\star$,
  \begin{enumerate}[(I)]
  \item \label{item:weak-convergence-h} for any continuous bounded function $h$, $\lim_n \pi_{\t_n(\omega)}(h) = \pi_{\t_\star}(h)$,
  \item \label{item:H_t_n-equicontinuous} the set $\{H_{\t_n(\omega)}, n \geq 0
    \}$ is equicontinuous,
  \item \label{item:sup_n pi_t_n_omega} $\sup_n \pi_{\t_n(\omega)}\left(|H_{\t_n(\omega)}|^{1/(2\alpha)}\right) < +\infty$,
  \item \label{item:lim_n_H_t_n_omega} $\lim_n H_{\t_n(\omega)}(x) = H_{\t_\star}(x)$ for any $x \in \Xset$,
  \item  \label{item:pi_t_star} $\pi_{\t_\star}(|H_{\t_\star}|) < +\infty$.
  \end{enumerate}
  The proof is then concluded by application of
  Lemma~\ref{lem:mun:fn:convergence}.

\begin{proof}[Proof of \eqref{item:weak-convergence-h}]
  Under the conditions I\ref{EES:Pi} and
  I\ref{E:KernelP}\ref{E:Irred}-\ref{E:Drift}-\ref{E:Minor},
  I\ref{E:ProcY}\ref{E:ProcY:item1}-\ref{E:ProcY:item2} and $\PE[V(X_0)]< +
  \infty$, \cite[Proposition 3.3.]{fort:moulines:priouret:2010} proves that
  this condition holds for any $\omega \in \Omega_1$ such that $\PP(\Omega_1)=1$.
\end{proof}

\begin{proof}[Proof of \eqref{item:H_t_n-equicontinuous}]
  Let $C_\t, \rho_\t$ be given by P[$\alpha$]. For any constants $C,v>0$ and
  $\rho \in (0,1)$, set
  \begin{equation}
    \label{eq:Tset:restricted}
    \Tset_{C,\rho,v} \eqdef \{ \t \in \Tset \ \text{s.t.}  \ C_\t \leq C,
  \rho_\t \leq \rho, \t(V) \leq v\} \eqsp.
  \end{equation}
  By (\ref{eq:IT:LthetaCvg}) and
  I\ref{E:ProcY}-\ref{E:ProcY:item2},
\begin{equation}\label{eq:IT:controleergo}
\limsup_n C_{\t_n} < +\infty \eqsp, \PP-\as \qquad \limsup_n \rho_{\t_n} < 1
\eqsp,  \PP-\as
\end{equation}  and $\limsup_n \t_n(V) < +\infty$ w.p.1.  Therefore, it is
sufficient to prove that the set $\{ H_\t, \t \in \Tset_{C,\rho,v} \}$ is
equicontinuous.

Let $C, v>0$ and $\rho \in(0,1)$ be fixed.  Observe that by definition of
$L_\t$ (see (\ref{eq:DefinitionLtheta})) and  (\ref{eq:Controle:FctPoisson})
  \begin{equation}
    \label{eq:IT:equicontinuite:tool1}
 \sup_{\t \in  \Tset_{C,\rho,v}}  \fnorm{\operpoisson_\t f}{V^\alpha} \leq \fnorm{f}{V^\alpha} \ \left(C \vee
  (1-\rho)^{-1} \right)^2 \eqsp.
  \end{equation}
  For any $x,x' \in \Xset$ and any $\t \in \Tset_{C,\rho,v}$,
  \begin{multline*}
    \left| H_\t(x) - H_\t(x') \right| \leq \left| \operpoisson_\t f(x) +
      \operpoisson_{\t} f(x') \right|\left| \operpoisson_\t f(x)
      -  \operpoisson_{\t} f(x') \right| \\
    + \left|P_\t \operpoisson_\t f(x) + P_\t \operpoisson_\t f(x') \right|
    \left(\left| \operpoisson_\t f(x) - \operpoisson_\t f(x') \right| + \left| f(x) - f(x') \right| \right)  \eqsp,
  \end{multline*}
  where we have used $P_\theta \operpoisson_\t f(x) - P_\theta \operpoisson_\t f(x') = (\operpoisson_\t - \Id) [f(x) - f(x')] $.
  By \eqref{eq:prop:A1:IT:drift} and \eqref{eq:IT:equicontinuite:tool1},
  for any $\t \in \Tset_{C,\rho,v}$,
\begin{equation}
\label{eq:IT:equicontinuite:tool2}
\left|P_\t \operpoisson_\t f(x) \right| \leq
\fnorm{f}{V^\alpha} \ \left(C \vee (1-\rho)^{-1} \right)^2 \left( V(x) +  \tilde b v \right) \eqsp.
\end{equation}
Therefore, since $f$ and $V$ are continuous, it suffices to prove that the set
$\{\operpoisson_{\t} f, \t \in \Tset_{C,\rho,v} \}$ is equicontinuous.
Lemma~\ref{lem:equicontinuity:Ptheta} and
I\ref{E:KernelP}-\ref{E:equicontinuite} imply that the set $\{P_\t f, \t \in
\Tset_{C,\rho,v} \}$ is equicontinuous. Repeated applications of this Lemma
shows that for any $\ell \geq 1$, the set $\{P_\t^\ell f, \t \in
\Tset_{C,\rho,v} \}$ is equicontinuous.  By Proposition~\ref{prop:A1:IT}, we
have for any $\t \in \Tset_{C,\rho,v}$,
\begin{multline*}
  \left| \operpoisson_\t f(x) - \operpoisson_\t f(x') \right| \leq
  \sum_{k=0}^{n-1}
  \left| P_\t^k f(x) - P_\t^k f(x') \right| \\
  + 2 \rho^n \, C \, (1-\rho)^{-1} \ \left(V(x) + V(x')\right) \eqsp.
\end{multline*}
Then, the set $\{\operpoisson_\t f, \t \in \Tset_{C,\rho,v} \}$ is
equicontinuous.
\end{proof}

\begin{proof}[Proof of \eqref{item:sup_n pi_t_n_omega}] By (\ref{eq:IT:LthetaCvg}) and (\ref{eq:limsuppitheta}),  there exists $\Omega_3$ such that $\PP(\Omega_3) =1$
  and for any $\omega \in \Omega_3$, \eqref{item:sup_n pi_t_n_omega} holds if,
  for any constants $C,v>0$ and $\rho \in (0,1)$,
\[
\sup_{\t \in \Tset_{C,\rho,v}} \fnorm{H_{\t}}{V^{2\alpha}}^{1/2\alpha} \,
\pi_\t(V) < +\infty \eqsp,
\]
where $\Tset_{C,\rho,v}$ is defined by (\ref{eq:Tset:restricted}).  The bound
on $H_\t$ follows from \eqref{eq:IT:equicontinuite:tool1} and
\eqref{eq:IT:equicontinuite:tool2}. By Lemma~\ref{lem:BoundCandRho} and
Proposition~\ref{prop:A1:IT}, $\sup_{\t \in \Tset_{C,\rho,v}} \pi_\t(V) \leq
(1-\tilde \lambda)^{-1} \ \tilde b \ v$.


\end{proof}

\begin{proof}[Proof of~\eqref{item:lim_n_H_t_n_omega}]
We first prove that for any $x
\in \Xset$, $\lim_n \operpoisson_{\t_n} f(x) \aslim \operpoisson_{\t_\star}
f(x)$.  By Proposition~\ref{prop:A1:IT},
for any $\ell \geq 1$,
\begin{multline}
\label{eq:IT:tool3}
  \left| \operpoisson_{\t_n} f(x) - \operpoisson_{\t_\star} f(x) \right| \leq
  C_{\t_n} \rho_{\t_n}^\ell V(x) + C_{\t_\star} \rho_{\t_\star}^\ell V(x) +
  \left|
    \pi_{\t_n}(f) - \pi_\star(f) \right|  \\
  + \sum_{k=0}^{\ell-1} \left| P_{\t_n}^k f(x) - P_{\t_\star}^k f(x) \right|
  \eqsp.
\end{multline}
From \cite[Proposition 3.3.]{fort:moulines:priouret:2010}, $\pi_{\t_n}(f)-
\pi_{\t_\star}(f) \aslim 0$ since $f$ is continuous. In addition, following the
same lines as in the proof of \cite[Proposition 3.3, Lemma
4.4.]{fort:moulines:priouret:2010}, it holds $P_{\t_n}^k f(x) - P_{\t_\star}^k
f(x) \aslim 0$ for any $k$. Therefore, by (\ref{eq:IT:controleergo}),
\eqref{eq:IT:tool3} shows that $\lim_n \operpoisson_{\t_n} f(x) \aslim
\operpoisson_{\t_\star} f(x)$.

It remains to prove that $\lim_n P_{\t_n} \operpoisson_{\t_n} f(x) \aslim
P_{\t_\star} \operpoisson_{\t_\star} f(x)$. This is a consequence of the above discussion and the equality
\[
P_{\t_n} \operpoisson_{\t_n} f(x) - P_{\t_\star} \operpoisson_{\t_\star} f(x) =
\operpoisson_{\t_n} f(x) - \operpoisson_{\t_\star} f(x) +\pi_{\t_n}(f) - \pi_{\t_\star}(f) \eqsp,
\]
which follows from~(\ref{eq:PoissonEquation}).

Combining the two results above, for any $x \in \Xset$, $H_{\t_n}(x) \aslim H_{\t_\star}(x)$ as $n
\to +\infty$. Since $\Xset$ is Polish, there exists a countable dense subset
$\mathcal{D}$ of $\Xset$ and a set $\Omega_4$ with $\PP(\Omega_4) =1$ such that
for any $\bar x \in \mathcal{D}$ and any $\omega\in \Omega_4$,
\[
\lim_n H_{\t_n(\omega)}(\bar x) = H_{\t_\star}(\bar x) \eqsp.
\]
The proof is concluded by the inequality
\begin{multline*}
  \left| H_{\t_n(\omega)}(x) - H_{\t_\star}( x) \right| \leq \left|
    H_{\t_n(\omega)}(x) - H_{\t_n(\omega)}(\bar x) \right| \\
  + \left| H_{\t_n(\omega)}(\bar x)-H_{\t_\star}(\bar x) \right| +
  \left|H_{\t_\star}(\bar x) - H_{\t_\star}( x) \right| \eqsp,
\end{multline*}
the continuity of $H_{\t_\star}$  and \eqref{item:H_t_n-equicontinuous}.
\end{proof}

\begin{proof}[Proof~\eqref{item:pi_t_star}]
  Since $H_{\t_\star} \in \L_{V^{2\alpha}}$, this is a consequence of
  Lemma~\ref{lem:BoundCandRho} and Assumption
  I\ref{E:ProcY}-\ref{E:ProcY:item1}.
\end{proof}
\end{proof}

\subsubsection{Proof of Theorem~\ref{theo:IT}}
\label{sec:proof:theo:IT}
We check the conditions of Theorem~\ref{theo:TCL}.
\textbf{A\ref{hyp:geometric-ergodicity}} to \textbf{A\ref{hyp:X_conditionellement_markov}} hold
(see Propositions~\ref{prop:A1:IT} and \ref{prop:A3:A4:ICMCM}) and we now
prove \textbf{A\ref{hyp:variance:TCLsurPi}}. We first check condition
\textbf{A\ref{hyp:variance:TCLsurPi}}-\ref{hyp:variance:TCLsurPi:lineaire}.  For any
function $f \in \L_{V^\alpha} \cap \mathcal{N}$, define
\begin{equation}
  \label{eq:fonction:Gf:EE}
  G_f(z)  \eqdef \epsilon \iint \left( \delta_z(\rmd z') - \t_\star(\rmd
  z') \right) \pi_{\t_\star}(\rmd x)\ratio(x,z') \left(
  \operpoisson_{\t_\star}f(z') - \operpoisson_{\t_\star}f(x) \right) \eqsp.
\end{equation}
Let $f \in \L_{V^\alpha} \cap \mathcal{N}$; note that $G_f \in \L_{V^\alpha}$.
Recall that by Eq. \eqref{eq:definition-Pt-EE}, for any $\t$ such that
$\t(V^\alpha)<+\infty$,
\begin{equation}
  \label{eq:EE:variationP}
  P_\t f(x) - P_{\t_\star}f(x) = \epsilon \int \left[ \t(\rmd y) -\t_\star(\rmd y) \right] \ratio(x,y) \left( f(y) -f(x)
\right)  \eqsp.
\end{equation}
Then, using (\ref{eq:fonction:Gf:EE}),
\begin{multline*}
  \pi_{\t_\star} \left(P_{\t_k} - P_{\t_\star} \right) \operpoisson_{\t_\star}
  f  \\
  = \epsilon \iint \pi_{\t_\star}(\rmd x) \left[ \t_k(\rmd z) -\t_\star(\rmd z) \right]
  \ratio(x,z) \left[ \operpoisson_{\t_\star}f(z) -
    \operpoisson_{\t_\star}f(x) \right] = \theta_k(G_f) \eqsp.
\end{multline*}
Therefore,
\begin{multline*}
  \frac{1}{\sqrt{n}} \sum_{k=1}^n \pi_{\t_\star} \left(P_{\t_k} - P_{\t_\star}
  \right) \operpoisson_{\t_\star} f = \frac{1}{n} \sum_{k=1}^n
  \frac{n}{k} \frac{1}{\sqrt{n}} \sum_{j=1}^k G_f(Y_j) \\
  = \int_0^1 t^{-1} S_n(t) \rmd t + \sum_{k=1}^{n-1} \int_{k/n}^{(k+1)/n}
  \left( \frac{n}{k} - \frac{1}{t} \right) S_n(t) \rmd t + \frac{1}{n} S_n(1)
  \eqsp,
\end{multline*}
with $S_n(t) \eqdef n^{-1/2}\sum_{j=1}^{[nt]} G_f(Y_j)$. Note that
\[
\PE\left[ \left| \sum_{k=1}^{n-1} \int_{k/n}^{(k+1)/n} \left( \frac{n}{k} - \frac{1}{t}
  \right) S_n(t) \rmd t \right| \right] \leq \frac{1}{\sqrt n} \sum_{k=1}^n
\frac{1}{k+1} \frac{1}{k} \sum_{j=1}^k \PE\left[ \left| G_f(Y_j) \right|  \right] \eqsp.
\]
Since $G_f \in \L_{V^\alpha}$, I\ref{E:ProcY}-\ref{E:ProcY:item1} implies that
$\sup_{k \geq 0} \PE[|G_f|(Y_k)] < \infty$.  Therefore,
$$
\sum_{k=1}^{n-1} \int_{k/n}^{(k+1)/n} \left( \frac{n}{k} - \frac{1}{t} \right) S_n(t) \rmd t + \frac{1}{n} S_n(1) \plim 0 \eqsp.
$$
Using I\ref{E:ProcY}-\ref{E:ProcY:item3} and the Continuous mapping Theorem (\cite[Theorem 1.3.6]{vandervaart:wellner:1996}),
we obtain
\[
  \frac{1}{\sqrt{n}} \sum_{k=1}^n \pi_{\t_\star} \left(P_{\t_k} - P_{\t_\star}
  \right) \operpoisson_{\t_\star} f \dlim \tilde \gamma^2(f) \int_0^1 t^{-1} B_t \rmd t \eqsp.
\]
Since $\int_0^1 t^{-1} B_t \rmd t = \int_0^1 \log(t) \rmd B_t$, $\int_0^1 t^{-1} B_t \rmd t$ is a Gaussian random variable with zero mean and variance $\int_0^1 \log^2(t) \rmd t = 2$.

\bigskip

We now check condition
\textbf{A\ref{hyp:variance:TCLsurPi}}-\ref{hyp:variance:TCLsurPi:reste}. Note that
\[
n^{-1/2} \sum_{k=1}^n \pi_{\t_k} \left(P_{\t_k} - P_{\t_\star} \right)
\Lambda_{\t_\star} \left(P_{\t_k} - P_{\t_\star} \right) \Lambda_{\t_\star} f
=n^{-1/2} \sum_{k=1}^n \pi_{\t_k} (G^{f}_{\t_k})  \eqsp,
\]
where
\begin{equation}
\label{eq:definition-G-theta}
G^{f}_{\t} (x) \eqdef \left(P_{\t} - P_{\t_\star} \right) \Lambda_{\t_\star}
\left(P_{\t} - P_{\t_\star} \right) \Lambda_{\t_\star} f(x) \eqsp.
\end{equation}
We write for any $x \in \Xset$ and any $\ell_k \in \Nset$,
\[
\pi_{\t_k}(G^{f}_{\t_k}) = \left(\pi_{\t_k} -P_{\t_k}^{\ell_k} \right)G^{f}_{\t_k} (x) +
\left( P_{\t_k}^{\ell_k} - P_{\t_\star}^{\ell_k} \right)G^{f}_{\t_k} (x) +
P_{\t_\star}^{\ell_k}G^{f}_{\t_k} (x) \eqsp.
\]
By Proposition~\ref{prop:A1:IT}, P[$\alpha$] holds and there exist $C_\t,
\rho_\t$ such that $\Vnorm{P_\t^n - \pi_\t}{V^\alpha} \leq C_\t \rho_\t^n$.
Furthermore, Lemma~\ref{lem:explicit:control} and
I\ref{E:ProcY}\ref{E:ProcY:item2} imply that $\limsup_n C_{\t_n} < +\infty$
w.p.1. and there exists a constant $\rho \in (0,1)$ such that $\limsup_n
\rho_{\t_n} \leq \rho$, w.p. 1.  Set $\ell_k \eqdef \lfloor \ell \ln k \rfloor$
with $\ell$ such that $1/2+\ell \ln \rho < 0$. Let $x \in \Xset$ be fixed.

By Lemma~\ref{lem:prop:EE:remainder:term1} and
I\ref{E:ProcY}-\ref{E:ProcY:item2}, there exists an almost surely finite random
variable $C_1$ s.t.
\[
\left| \frac{1}{\sqrt{n}} \sum_{k=1}^n\left(\pi_{\t_k} -P_{\t_k}^{\ell_k}
  \right)G^{f}_{\t_k}(x) \right| \leq C_1  V^{\alpha}(x)   n^{-1/2} \sum_{k=1}^n
 \rho^{\ell_k} \eqsp.
\]
Since $n^{-1/2} \sum_{k=1}^n \rho^{\ell_k} \leq\rho^{-1} n^{-1/2} \sum_{k=1}^n k^{\ell \ln \rho} \to_{n
  \to \infty} 0$,  it holds
\[
\frac{1}{\sqrt{n}} \sum_{k=1}^n\left(\pi_{\t_k} -P_{\t_k}^{\ell_k}
\right)G^{f}_{\t_k}(x) \aslim 0 \eqsp.
\]

By Lemma~\ref{lem:prop:EE:remainder:term2}, there exist some positive constants
$C_2, \kappa_\star, a$ such that
\[
\PE\left[\left( \sum_{k=1}^n \{P_{\t_k}^{\ell_k} - P_{\t_\star}^{\ell_k} \}
    G^{f}_{\t_k}(x) \right)^2 \right]^{1/2} \leq C_2 \fnorm{f}{V^\alpha}
V^\alpha(x) \ \sum_{k=1}^n \frac{1}{k} \sum_{t=1}^{\ell_k-1} \left(\frac{\kappa_\star
  \ell_k }{k^{1/(2a)}} \right)^{at} \eqsp.
\]
Since $\lim_k \ell_k^a / k^{1/2} =0$, there exists $k_\star$ such that for $k \geq
k_\star$, $(\kappa_\star \ell_k)^a/k^{1/2} \leq 1/2$. Then,
\[
\frac{1}{\sqrt{n}} \sum_{k=1}^n \frac{1}{k} \sum_{t=1}^{\ell_k}
\left(\frac{\kappa_\star \ell_k}{k^{1/(2a)}} \right)^{at} \leq \frac{1}{\sqrt{n}}
\sum_{k=1}^{k_\star} \frac{1}{k} \sum_{t=1}^{\lceil \ell \ln k \rceil}
\left(\frac{\kappa_\star \ell_k}{k^{1/(2a)}} \right)^{at} + \frac{2}{\sqrt{n}} \sum_{k =
  k_\star+1}^n \frac{1}{k} \eqsp.
\]
The RHS tends to zero when $n \to +\infty$, which proves that $n^{-1/2}
\sum_{k=1}^n \{P_{\t_k}^{\ell_k} - P_{\t_\star}^{\ell_k} \} G^{f}_{\t_k}(x)
\plim 0$.

Finally, by Lemma~\ref{lem:prop:EE:remainder:term3}, there exists a constant
$C_3$ such that
\[
\PE\left[\left( \frac{1}{\sqrt{n}} \sum_{k=1}^n
    P_{\t_\star}^{\ell_k}G^{f}_{\t_k}(x) \right)^2 \right]^{1/2} \leq C_3
V^\alpha(x) \ \frac{1}{\sqrt{n}} \sum_{k=1}^n \frac{\ell_k^\alpha }{k} \to_{ n \to
  \infty} 0\eqsp,
\]
thus implying that $n^{-1/2} \sum_{k=1}^n P_{\t_\star}^{\ell_k}G^{f}_{\t_k}(x)
\plim 0$.

\begin{lemma}
  \label{lem:prop:EE:remainder:term1}
  Assume I\ref{EES:Pi} and
  I\ref{E:KernelP}\ref{E:Irred}-\ref{E:Drift}-\ref{E:Minor}.  Let $\alpha \in
  (0,1/2)$.  For any $f \in \L_{V^\alpha}$ and $\t\in \Tset$,
\[
G^{f}_\t(x) = \int (\t - \t_\star)^{\otimes 2}(\rmd z_{1:2}) \ F^{(0)}(x,z_1,z_2)
\eqsp,
\]
where $G_\t^f$ is defined by (\ref{eq:definition-G-theta}); and there exists a
constant $C$ such that for any $x \in \Xset$,
\[
\left| F^{(0)}(x,z_1,z_2) \right| \leq C \fnorm{f}{V^\alpha} \ V^{\alpha \wedge
  (\beta/ \tau)}(x) \ \left(V^\alpha(z_1) + V^\alpha(z_2) \right) \eqsp.
\]
In addition, there exists some constant $C'$ such that for any $\ell \in
\Nset$, any $\t \in \Tset$ and any $f \in \L_{V^\alpha}$,
\[
\fnorm{ \left(\pi_{\t} -P_{\t}^{\ell} \right) G^{f}_{\t}}{V^\alpha} \leq C'
\fnorm{f}{V^\alpha} \ \Vnorm{P_\t^\ell - \pi_\t}{V^\alpha} \ \t(V^\alpha)
\eqsp.
\]
\end{lemma}
\begin{proof}
  Set $\gamma \eqdef \alpha \wedge (\beta/ \tau)$. Throughout this proof, let
  $L_\t$ be the constant given by P[$\gamma$]. We have
\begin{multline*}
\label{eq:definition-F-0}
F^{(0)}(x,z_1,z_2) \eqdef    \epsilon^2  \ratio(x,z_2) \left[ \int \operpoisson_{\t_\star}(z_2,\rmd y)
  \ratio(y,z_1) \left(\operpoisson_{\t_\star} f(z_1) -\operpoisson_{\t_\star} f(y) \right) \right. \\
  \left. -  \int \operpoisson_{\t_\star}(x,\rmd y) \ratio(y,z_1) \left(\operpoisson_{\t_\star} f(z_1)
    -\operpoisson_{\t_\star} f(y) \right) \right] \eqsp.
\end{multline*}
Note that $\fnorm{r(\cdot, z_1)}{V^\gamma} \leq 1$ for any $z_1$ so that by
(\ref{eq:Controle:FctPoisson}),
\[
\left| \int \operpoisson_{\t_\star}(z_2,\rmd y) \ratio(y,z_1)
  \operpoisson_{\t_\star} f(z_1) \right| \leq L_{\t_\star}^4
\fnorm{f}{V^\alpha} V^\alpha(z_1) \  V^\gamma(z_2) \eqsp.
\]
In addition, since $\gamma - \beta/\tau\leq 0$, we have by definition of the
acceptance ratio $r$ (see (\ref{eq:definition-alpha}))
\[
r(x,z_2) V^\gamma(z_2) \leq V^\gamma(x) \eqsp.
\]
Then, there exists a constant $C$ such that
\[
\epsilon^2 \ratio(x,z_2) \, \left| \int \operpoisson_{\t_\star}(z_2,\rmd y)
  \ratio(y,z_1) \operpoisson_{\t_\star} f(z_1) \right| \leq C
\fnorm{f}{V^\alpha} V^\alpha(z_1) \ V^\gamma(x) \eqsp.
\]
Similar upper bounds can be obtained for the three remaining terms in
$F^{(0)}$, thus showing the upper bounds on $F^{(0)}$.

In addition, by P[$\gamma$]
\[
\fnorm{\left(\pi_{\t} -P_{\t}^{\ell} \right) G^{f}_{\t} f(x)}{V^\alpha} \leq
\Vnorm{\pi_{\t} -P_{\t}^{\ell}}{V^\alpha} \ \ \fnorm{G^{f}_{\t} f}{V^\alpha}
V^\alpha(x)\eqsp.
\]
The proof is concluded upon noting that $|G_\t^f(x)| \leq C \fnorm{f}{V^\alpha}
\t(V^\alpha)$.
\end{proof}

\begin{lemma}
\label{lem:expression-F}
Assume I\ref{EES:Pi} and
I\ref{E:KernelP}\ref{E:Irred}-\ref{E:Drift}-\ref{E:Minor}. Let $\alpha \in
(0,1/2)$. There exist some constants $C,\kappa_\star$ and $\rho_\star \in
(0,1)$ such that for any $t \geq 1$, any integers $u_1, \cdots, u_t$ and any $f
\in \L_{V^\alpha}$,
  \begin{multline*}
    \left(P_{\t} - P_{\t_\star} \right) \left( P^{u_t}_{\t_\star} -
      \pi_{\t_\star} \right) \cdots \left(P_{\t} - P_{\t_\star} \right)
    \left( P^{u_1}_{\t_\star} - \pi_{\t_\star} \right)  G^{f}_{\t} (x) \\
    = \idotsint \left( \t - \t_\star \right)^{\otimes (t+2)}(\rmd z_{1:t+2}) \
    F^{(t)}_{u_{1:t}}(x,z_1,\cdots, z_{t+2})
  \end{multline*}
  where $G_\t^f$ is defined in (\ref{eq:definition-G-theta}), and
\begin{equation} \label{eq:bound-F}
\left| F^{(t)}_{u_{1:t}}(x,z_1,\cdots, z_{t+2})  \right|\leq
C \fnorm{f}{V^\alpha} \kappa_\star^t \ \rho_{\star}^{\sum_{j=1}^t u_j}  V^{\alpha \wedge (\beta/\tau)}(x)  \  \sum_{j=1}^{t+2} V^\alpha(z_j) \eqsp.
\end{equation}
\end{lemma}
\begin{proof}
  By repeated applications of Eq.~\eqref{eq:EE:variationP}, it can be proved
  that the functions $F^{(t)}_{\chunk{u}{1}{t}}$ are recursively defined as
  follows
\begin{multline}
    \label{eq:Induction:Ft}
    F_{u_{1:t}}^{(t)}(x, z_1, \cdots, z_{t+2}) \eqdef \epsilon
    \ratio(x,z_{t+2}) \times \\
    \int \left( P^{u_t}_{\t_\star} (z_{t+2},\rmd y) - P^{u_t}_{\t_\star}(x,\rmd
      y) \right) F_{u_{1:t-1}}^{(t-1)}(y, z_1, \cdots, z_{t+1}) \eqsp,
 \end{multline}
 where $F_{u_{1:0}}^{(0)} = F^{(0)}$ and $F^{(0)}$ is given by
 Lemma~\ref{lem:prop:EE:remainder:term1}.

 The proof of the upper bound is by induction.  The property holds for $t=1$.
 Assume it holds for $t \geq 2$.  Set $\gamma \eqdef \alpha \wedge (\beta/
 \tau)$; by Proposition~\ref{prop:A1:IT} and the property P[$\gamma$], there
 exist some constants $C_\star$ and $\rho_\star \in (0,1)$ such that
 $\Vnorm{P_{\t_\star}^\ell - \pi_{\t_\star}}{V^\gamma} \leq C_{\t_\star}
 \rho_{\t_\star}^\ell$.  Then,
\begin{align*}
  &  \left|F_{u_{1:t}}^{(t)}(x, z_{1:t+2}) \right| \leq C \fnorm{f}{V^\alpha} \, \kappa_\star^{t-1}  \rho_{\t_\star}^{\sum_{j=1}^{t-1} u_j}  \ \left(\sum_{j=1}^{t+1} V^\alpha(z_j) \right)\\
  & \phantom{\left|F_{u_{1:t}}^{(t)}(x, z_{1:t+2}) \right| \leq} \times
  \ratio(x,z_{t+2}) \left[ \Vnorm{P^{u_t}_{\t_\star} -
      \pi_{\t_\star}}{V^\gamma} V^\gamma(z_{t+2})  + \Vnorm{P^{u_t}_{\t_\star} - \pi_{\t_\star}}{V^\gamma} V^\gamma(x)\right]  \\
  & \quad \leq C \ \fnorm{f}{V^\alpha} \ \kappa_\star^{t-1} \ \epsilon \
  C_{\t_\star} \ \rho_{\t_\star}^{ \sum_{j=1}^t u_{j}} \ \ratio(x,z_{t+2}) \
  \left\{ V^\gamma(z_{t+2}) + V^\gamma(x)\right\} \eqsp.
\end{align*}
Since $\gamma \leq \beta /\tau$, $\ratio(x,z_{t+2}) V^\gamma(z_{t+2}) \leq
V^\gamma(x)$ thus showing \eqref{eq:bound-F} with $\kappa_\star = 2
C_{\t_\star} \epsilon$.
\end{proof}

\begin{lemma}
  \label{lem:prop:EE:remainder:term2}
  Assume I\ref{E:BoundPi},
  I\ref{E:KernelP}\ref{E:Irred}-\ref{E:Drift}-\ref{E:Minor} and I\ref{E:ProcY}.
  Let $\alpha \in (0, 1/2)$. There exist positive constants $C, \kappa, a$ such
  that for any $f \in \L_{V^\alpha}$, any $k,\ell \geq 1$ and any $x \in
  \Xset$,
\[
\PE\left[ \left( \left\{ P_{\t_k}^{\ell} - P_{\t_\star}^{\ell} \right\}
    G^{f}_{\t_k} (x) \right)^2 \right]^{1/2} \leq C \fnorm{f}{V^\alpha}
\frac{V^\alpha(x)}{k} \sum_{t=1}^{\ell-1} \left( t \kappa k^{-1/(2a)}
\right)^{at} \eqsp,
\]
where $G_\t^f$ is given by (\ref{eq:definition-G-theta}).
\end{lemma}
\begin{proof}
  For any $g \in \L_{V^\alpha}$, $k, \ell \geq 1$ and $x \in \Xset$,
  \begin{align*}
    &P_{\t_k}^{\ell} g(x) - P_{\t_\star}^{\ell} g(x) \\
    &= \sum_{t=1}^{\ell-1} \sum_{u_{1:t} \in \mathcal{U}_t} P_{\t_\star}^{\ell
      -t - \sum_{j=1}^t u_j} \left(P_{\t_k} - P_{\t_\star} \right)
    P^{u_t}_{\t_\star} \cdots \left(P_{\t_k} - P_{\t_\star} \right)
    P^{u_1}_{\t_\star} g(x) \eqsp, \\
    &= \sum_{t=1}^{\ell-1} \sum_{u_{1:t} \in \mathcal{U}_t} P_{\t_\star}^{\ell
      -t - \sum_{j=1}^t u_j} \left(P_{\t_k} - P_{\t_\star}
    \right) \left( P^{u_t}_{\t_\star} - \pi_{\t_\star} \right) \\
    & \phantom{= \sum_{t=1}^{\ell-1} \sum_{u_{1:t} \in \mathcal{U}_t}
      P_{\t_\star}^{\ell -t - \sum_{j=1}^t u_j}} \times \cdots \left(P_{\t_k} -
      P_{\t_\star} \right) \left( P^{u_1}_{\t_\star} - \pi_{\t_\star} \right)
    g(x) \eqsp,
  \end{align*}
  where $\mathcal{U}_t =\{ u_{1:t}, u_j \in \Nset, \sum_{j=1}^t u_j \leq \ell -t
  \}$. Fix $t \in \{1, \cdots, \ell-1 \}$ and $u_{1:t} \in \mathcal{U}_t$. Then by Lemma~\ref{lem:expression-F},
  \begin{multline*}
    P_{\t_\star}^{\ell -t - \sum_{j=1}^t u_j} \left(P_{\t_k} - P_{\t_\star}
    \right) \left( P^{u_t}_{\t_\star} - \pi_{\t_\star} \right) \cdots
    \left(P_{\t_k} - P_{\t_\star} \right) \left( P^{u_1}_{\t_\star} -
      \pi_{\t_\star} \right) G^{f}_{\t_k}(x) \\
    = \int \left( \t_k -\t_\star\right)^{\otimes(t+2)}(\rmd z_{1:t+2}) \int
    P_{\t_\star}^{\ell -t - \sum_{j=1}^t u_j}(x,\rmd y) F_{u_{1:t}}^{(t)}(y,
    z_1, \cdots, z_{t+2}) \eqsp.
  \end{multline*}
  Assumptions I\ref{E:ProcY}-\ref{E:ProcY:item2} and
  I\ref{E:ProcY}-\ref{E:ProcY:item4} and Lemma~\ref{lem:expression-F} show that
  there exist constants $C, \kappa_\star,\rho_\star \in (0,1)$ such that
  \begin{multline*}
    \left\| \int \left( \t_k -\t_\star\right)^{\otimes(t+2)}(\rmd z_{1:t+2})
      \int P_{\t_\star}^{\ell -t - \sum_{j=1}^t u_j}(x,\rmd y)
      F_{u_{1:t}}^{(t)}(y,
      z_1, \cdots, z_{t+2}) \right\|_2 \\
    \leq \frac{C}{k^{1 + t/2}} A_t \fnorm{f}{V^\alpha} \kappa_\star^t \
    \rho_{\star}^{\sum_{j=1}^t u_j} \ P_{\t_\star}^{\ell -t - \sum_{j=1}^t
      u_j}V^\alpha(x) \, \eqsp.
  \end{multline*}
  Finally, Proposition~\ref{prop:A1:IT} implies that $\sup_{j \geq 0}
  \fnorm{P_{\t_\star}^{j } V^\alpha}{V^\alpha} < +\infty$.  By combining these
  results, we have for some constant $C$
\[
\left\| P_{\t_k}^{\ell} G^{f}_{\t_k}(x) - P_{\t_\star}^{\ell} G^{f}_{\t_k}(x)
\right\|_2 \leq C k^{-1} \fnorm{f}{V^\alpha} V^\alpha(x) \sum_{t=1}^{\ell-1}
A_t \kappa_\star^t  \, k^{-t/2} \sum_{u_{1:t} \in
  \mathcal{U}_t}\rho_{\star}^{\sum_{j=1}^t u_j} \eqsp.
\]
Note that $\sum_{u_{1:t} \in \mathcal{U}_t}\rho_{\star}^{\sum_{j=1}^t u_j} \leq
(1-\rho_{\star})^{-t}$.  Furthermore, there exists $a>0$ such that $A_t \leq
t^{a t}$. Therefore,
\begin{multline*}
  \left\| P_{\t_k}^{\ell} G^{f}_{\t_k}(x) - P_{\t_\star}^{\ell} G^{f}_{\t_k}(x)
  \right\|_2  \\
  \leq C k^{-1} \fnorm{f}{V^\alpha} V^\alpha(x) \sum_{t=1}^{\ell-1} \left( t
    \kappa^{1/a}(1-\rho_\star)^{-1/a} \, k^{-1/(2a)} \right)^{a t} \eqsp.
\end{multline*}
This concludes the proof.
\end{proof}

\begin{lemma}
  \label{lem:prop:EE:remainder:term3}
  Assume I\ref{EES:Pi},
  I\ref{E:KernelP}\ref{E:Irred}-\ref{E:Drift}-\ref{E:Minor} and I\ref{E:ProcY}.
  Let $\alpha \in (0,1/2)$ and $f \in \L_{V^\alpha}$. Then, there exists a
  constant $C$ such that for any $k,\ell \geq 1$ and any $x \in \Xset$,
\[
\PE\left[ \left(P_{\t_\star}^{\ell} G^{f}_{\t_k} (x) \right)^2 \right]^{1/2} \leq
C \ \ell^\alpha \ \fnorm{f}{V^\alpha} \ k^{-1} V^{\alpha}(x) \eqsp.
\]
\end{lemma}
\begin{proof}
  We have \[ P_{\t_\star}^{\ell} G^{f}_{\t_k} (x) = \iint \left( \t_k -
    \t_\star \right)^{\otimes 2}(\rmd z_{1:2}) H_\ell(x,z_1,z_2) \eqsp,
  \]
  with $H_\ell(x,z_1,z_2) \eqdef  P_{\t_\star}^{\ell}(x,
  F^{(0)}(\cdot,z_1,z_2))$ where $F^{(0)}$ is given by
  Lemma~\ref{lem:prop:EE:remainder:term1}.
  Lemma~\ref{lem:prop:EE:remainder:term1} also implies that there exists a
  constant $C$ such that
  \begin{equation}
    \label{eq:lem:prop:EE:remainder:term3:tool1}
 \left| H_\ell(x,z_1,z_2)  \right| \leq C  \, \fnorm{f}{V^\alpha} \left(V^\alpha(z_1) +
  V^\alpha(z_2) \right) \ P_{\t_\star}^\ell V^\alpha(x) \eqsp.
 \end{equation}
 By I\ref{E:ProcY}, the variance of $P_{\t_\star}^{\ell} G^{f}_{\t_k} (x)$ is
 upper bounded by
\[
C \fnorm{f}{V^\alpha}^2 (P_{\t_\star}^\ell V^\alpha(x))^2 k^{-2} \eqsp.
\]  The proof is
concluded by application of the drift inequality~(\ref{eq:prop:A1:IT:drift})
and I\ref{E:ProcY}-\ref{E:ProcY:item1}.
\end{proof}

\section{Appendix}
\label{appendix}
\subsection{Technical lemmas}
The following lemma is (slightly) adapted from \cite[Lemma
4.2.]{fort:moulines:priouret:2010}
\begin{lemma}
\label{lem:regularity-in-theta:poisson}
Assume \textbf{A\ref{hyp:geometric-ergodicity}}. For any $f \in \L_{V^\alpha}$ and $\t,\t' \in \Theta$,
\begin{align*}
  &\Vnorm{\pi_\t - \pi_{\t'}}{V^{\alpha}} \leq 2 (L_{\t'} \vee L_{\t})^4 \pi_\t(V^{\alpha}) \ D_{V^{\alpha}}(\t,\t') \eqsp, \\
  &\Vnorm{ \operpoisson_\t - \operpoisson_{\t'} }{V^{\alpha}} \leq
  3   \ \left(L_\t \vee L_{\t'} \right)^6 \pi_\t(V^\alpha) D_{V^{\alpha}}(\t,\t') \\
  & \Vnorm{ P_\t \operpoisson_\t - P_{\t'} \operpoisson_{\t'} }{V^{\alpha}}
  \leq 5 \ \left(L_\t \vee L_{\t'} \right)^6 \pi_\t(V^\alpha)
  D_{V^{\alpha}}(\t,\t') \eqsp.
\end{align*}
where $L_\t$ and $\operpoisson_\t $ are given by \eqref{eq:DefinitionLtheta} and
(\ref{eq:operateur:poisson}).
\end{lemma}

The following lemma can be obtained from \cite{roberts:rosenthal:2004},
\cite{fort:moulines:2003:SPA}, \cite{douc:moulines:rosenthal:2004} or
\cite{baxendale:2005} (see also the proof of \cite[Lemma
3]{saksman:vihola:2010} for a similar result).
\begin{lemma}
\label{lem:explicit:control}
Let $\{P_\t, \t \in \Tset \}$ be a family of phi-irreducible and aperiodic
Markov kernels. Assume that there exist a function $V: \Xset \to \coint{1,+\infty}$, and for any $\t \in \Theta$ there exist some constants
$b_\t< \infty$, $\delta_\t \in (0,1)$, $\lambda_\t \in (0,1)$ and a probability
measure $\nu_\t$ on $\Xset$ such that
\begin{align*}
  & P_\t V \leq  \lambda_\t V + b_\t \eqsp, \\
  & P_\t(x,\cdot) \geq \delta_\t \ \nu_\t(\cdot) \ \un_{\{V \leq c_\t
    \}}(x) \qquad c_\t \eqdef 2 b_\t (1-\lambda_\t)^{-1} -1 \eqsp.
\end{align*}
Then there exists $\gamma>0$ and for any $\t$, there exist some finite
constants $C_\t$ and $\rho_\t \in (0,1)$ such that
$$\Vnorm{P_\t^n(x,\cdot) - \pi_\t}{V} \leq C_\t \ \rho_\t^n \ V(x)$$
and
\begin{equation*}
C_\t \vee (1-\rho_\t)^{-1} \leq C \left\{ b_\t \vee
  \delta_\t^{-1} \vee (1-\lambda_\t)^{-1} \right\}^\gamma \eqsp.
\end{equation*}
\end{lemma}

Lemma~\ref{lem:mun:fn:convergence} is proved in \cite[Section
4]{fort:moulines:priouret:supp:2010}.
\begin{lemma}
  \label{lem:mun:fn:convergence}
  Let $\Xset$ be a Polish space endowed with its Borel $\sigma$-field $\Xsigma$.
Let $\mu$ and $\sequence{mu}{n}$ be probability distributions on $(\Xset,
\Xsigma)$. Let $\sequence{h}{n}$ be an equicontinuous family of functions
from $\Xset$ to $\Rset$. Assume
\begin{enumerate}[(i)]
\item the sequence $\sequence{\mu}{n}$ converges weakly to $\mu$,
\item for any $x \in \Xset$, $\lim_{n \to \infty} h_n(x)$ exists, and there
exists $\gamma>1$ such that $\sup_n \mu_n(|h_n|^\gamma) + \mu(|\lim_{n}
h_n|) < +\infty$.
\end{enumerate}
Then, $\mu_n(h_n) \to \mu( \lim_{n } h_n)$.
\end{lemma}

\subsection{Weak law of large numbers for adaptive and interacting MCMC algorithms}
The proof of the theorem below is along the same lines as the proof of
\cite[Theorem 2.7]{fort:moulines:priouret:2010}, which addresses the strong law
of large numbers and details are omitted. Note that in this generalization, we
relax the condition $\sup_\t \fnorm{F(\cdot, \t)}{V} < +\infty$ of
\cite{fort:moulines:priouret:2010}. The proof is provided in the supplementary
paper~\citep{fort:moulines:priouret:vdk:2011-supplement}.

\newcounter{saveenum}
\begin{theo}
\label{theo:WeakLLN}
Assume \textbf{A\ref{hyp:DefinitionProc:X:Theta}}, \textbf{A\ref{hyp:geometric-ergodicity}} and
let $\gamma \in (0,1)$. Assume that
\begin{enumerate}[(i)]
\item \label{theo:WeakLLN:item1} $\limsup_{n \to \infty} L_{\t_n} < \infty$, $\PP$-\as\ where $L_\t$ is
  defined in Lemma~\ref{lem:BoundCandRho} applied with the closed interval
  $[\gamma, 1]$.
\item \label{theo:WeakLLN:item2} $\limsup_{n \to \infty} \pi_{\t_n}(V^\gamma) < \infty$, $\PP$-\as\ .
\item \label{theo:WeakLLN:item3} $\sup_{k \geq 1}  \PE\left[V(X_k)\right] < \infty$.
\item \label{theo:WeakLLN:item4} $n^{-1} \sum_{k=1}^n  D_{V^{\gamma}}(\t_k, \t_{k-1}) V^{\gamma}(X_k) \plim 0$.
\setcounter{saveenum}{\value{enumi}}
\end{enumerate}
Let $F :\Xset \times \Theta \to \Rset$ be a measurable function s.t.
\begin{enumerate}[(i)]
\setcounter{enumi}{\value{saveenum}}
\item \label{theo:WeakLLN:item5} $\limsup_{n \to \infty} \fnorm{F_{\t_n}}{V^{\gamma}} < +\infty$.
\item \label{theo:WeakLLN:item6} $n^{-1} \sum_{k=1}^{n-1} \fnorm{F_{\t_k} - F_{\t_{k-1}}}{V^{\gamma}} V^{\gamma}(X_k) \plim 0$.
\end{enumerate}
Then,
\[
\frac{1}{n} \sum_{k=0}^{n-1} F_{\t_k}(X_k) - \frac{1}{n} \sum_{k=0}^{n-1} \int
\pi_{\t_k}(\rmd x) F_{\t_k}(x)  \plim 0 \eqsp.
\]
\end{theo}


\end{document}